\documentclass[12pt,a4paper,twoside]{amsart}
\usepackage{a4wide, amsmath,amsthm, amsbsy,amsfonts,amssymb, txfonts, shuffle}
\usepackage{supertabular}
\usepackage[mathscr]{eucal}
\usepackage{stmaryrd,mathrsfs,graphicx, amscd, tikz-cd, cancel}
\usetikzlibrary{matrix,arrows,decorations.pathmorphing}
 \usepackage[normalem]{ulem}
\usepackage[
	linkcolor=blue,
	citecolor=red,
	urlcolor=red,
]{hyperref}
\usepackage[active]{srcltx}
\allowdisplaybreaks
\numberwithin{equation}{section}

\def\CC{{\mathbb C}}

\def\FF{{\mathbb F}} 
\def\GG{{\mathbb G}} 
\def\HH{{\mathbb H}}
\def\Hcal{{\mathbb H}}

\def\PP{{\mathbb P}}
 
\def\RR{{\mathbb R}} 
 
\def\ZZ{{\mathbb Z}}

\def\Ca{\mathrm{Ca}}

\def\G{\Gamma}
\def\g{\gamma}
\def\ct{{\rm cont}}

\def\GM{{\rm GM}}

\def\pr{{\rm pr}}

\def\bs{\backslash}
\newcommand{\eps}{\varepsilon}

\newcommand{\p}{\partial}

\def\Ccal{{\mathcal C}}

\def\Dcal{{\mathcal D}}

\def\Ecal{{\mathcal E}} 

\def\Fcal{{\mathcal F}} 
\def\Hcal{{\mathcal H}} 
\def\Hscr{{\mathscr H}} 
\def\Ical{{\mathcal I}} 
\def\Iscr{{\mathscr I}}
\def\Jcal{{\mathcal J}} 
\def\Kcal{{\mathcal K}}
\def\Lcal{{\mathcal L}}
\def\Lscr{{\mathscr L}}

\def\Ocal{{\mathcal O}}

\def\Pscr{{\mathscr P}}
\def\Qscr{{\mathscr  Q}}

\def\Scal{{\mathcal S}}

\def\Tcal{{\mathcal T}}
\def\Tscr{{\mathscr T}}
 
\def\Qscr{{\mathscr Q}} 
  
\def\Uscr{{\mathscr  U}}  
\def\Vcal{{\mathcal V}}

\def\Xcal{{\mathcal X}}

\def\la{\langle}
\def\ra{\rangle}
\def\half{{\tfrac{1}{2}}}

\def\glie{{\mathfrak{g}}}
\def\hlie{{\mathfrak{h}}}
\def\klie{{\mathfrak{k}}}

\def\sfrak{{\mathfrak{s}}}

\def\splie{{\mathfrak{sp}}}

\def\mfrak{{\mathfrak{m}}}
\def\cfrak{\mathfrak{c}}
\def\Dfrak{\mathfrak{D}}
\def\Sfrak{\mathfrak{S}}

\def\gfrak{\mathfrak{g}}

\def\nfrak{\mathfrak{n}}

\def\Rscr{\mathscr{R}}

\def\ssm{\smallsetminus}

\def\pt{{\scriptscriptstyle\bullet}}

\newcommand\Bir{\operatorname{Bires}}

\newcommand\conf{\mathscr{C}\!\mathit{onf}\!}

\newcommand\End{\operatorname{End}}

\newcommand\FB{\underline{\mathcal {F\!B}}}

\newcommand\Gr{\operatorname{gr}}

\newcommand\Hom{\operatorname{Hom}}

\newcommand\KS{\operatorname{KS}}

\newcommand\lie{\operatorname{lie}}

\newcommand\sym{\operatorname{Sym}}

\newcommand\Pol{\operatorname{Polyg}}

\newcommand\res{\operatorname{Res}}
\newcommand\sign{\operatorname{sgn}}
\newcommand\slin{\operatorname{\mathfrak{sl}}}

\newcommand\supp{\operatorname{supp}}

\newcommand\PBW{\operatorname{PBW}}
\newcommand\PSL{\operatorname{PSL}}
\newcommand\PGL{\operatorname{PGL}}
\newcommand\tr{\operatorname{Tr}}
\newcommand\U{\operatorname{U}}

\newcommand\Sp{\operatorname{Sp}}
\newcommand\Tr{\operatorname{Tr}}

\newcommand\Vir{\operatorname{Vir}}

\newtheorem{theorem}{Theorem}[section]
\newtheorem{lemma}[theorem]{Lemma}
\newtheorem{proposition}[theorem]{Proposition}
\newtheorem{corollary}[theorem]{Corollary}

\newtheorem{definition}{Definition}\numberwithin{definition}{section}

\theoremstyle{remark}
\newtheorem{example}[theorem]{Example}
\newtheorem{remark}[theorem]{Remark}

\newtheorem{question}[theorem]{Question}

\newcommand{\EL}[1]{\textcolor{red}{(#1)}}

\title{Bidifferentials, Lagrangian projections and the Virasoro extension}
\author{Eduard Looijenga}
\address{\vbox{\noindent Mathematisch Instituut, Universiteit Utrecht\newline Mathematics Department, University of Chicago}}
\begin{document}
\maketitle
\begin{abstract}  Let $C$ be a smooth projective curve over an algebraically closed field $k$ of characteristic  zero.
We prove  that a Lagrangian supplement of  $H^0(C, \Omega_C)$ in the de Rham cohomology group $H^1_{dR}(C)$
determines and is determined by a particular type of symmetric bidifferential on $C^2$ (its polar divisor must be twice the diagonal and have  biresidue one along it). When  $k=\CC$, a natural choice of such supplement is $H^{0,1}(C)$ and  
we show that this corresponds with the bidifferential that  after a twist is the rational $2$-form on $C^2$ found by  Biswas-Colombo-Frediani-Pirola. We determine the cohomology class carried by that $2$-form and define an analogue of this form as  rational $n$-form on $C^n$ that is regular on the $n$-point configuration space of $C$.

The proof relies on a local version of the above correspondence, which can be  stated in terms of  a complete discrete valuation ring.  We use this local version also to construct  in a \emph{natural} manner the Virasoro extension of the Lie algebra of derivations of a local field. 
\end{abstract}

\section{Introduction}
The goal of this article is to show that the three items mentioned in the title are intertwined (and that we might have included a fourth, namely a curious de Rham cohomology class on $C^2$, but the title is already too long). 
It is a spin-off of our approach to WZW theory, but as the results in question  have a semi-classical (if not classical) nature---these are rather basic properties of Riemann surfaces---they may  have an appeal beyond the 
WZW community. This also means that applications particular to WZW theory will not be discussed  here. It will however serve as a foundation for one or more subsequent papers that will  be devoted to it.

To explain what the paper is about, let $C$ be a smooth projective curve of genus $g$ over an algebraically closed field of characteristic zero.  Recall that we  have a short exact sequence involving its first de Rham cohomology group
\[
0\to H^0(C, \Omega_C)\to H^1_{dR}(C)\to H^1(C, \Ocal_C)\to 0
\]
and that the  middle term comes with a de Rham intersection pairing $H^1_{dR}(C)\times  H^1_{dR}(C)\to k$. 
It  is nondegenerate and alternating  and has  $H^0(C, \Omega_C)$ as a Lagrangian subspace. 
It is not hard to see that the Lagrangian supplements of $H^0(C, \Omega_C)$ in  $H^1_{dR}(C)$ form an affine  
space (torsor) whose translation space can be identified with  the vector space  $\sym^2 H^0(C, \Omega_C)$. 
We show that this affine space is canonically isomorphic with a space of bidifferentials, as defined below.

A \emph{bidifferential} $\eta$  of $C$is an object on $C^2$: at $p=(p_1, p_2)\in C^2$ it is  what looks like a $2$-form on $C^2$ at $p$, 
but is not quite that: if $z_i$ is a local coordinate of $C$ at $p_i$, then $\eta$ is of the form $fdz_1dz_2$ 
with $f\in\Ocal_{C^2,p}$, but where  we regard $dz_1dz_2$ as equal to $dz_2dz_1$, that is, 
$\eta$ is considered as an element  of $\sym^2_{\Ocal_{C^2,p}}\Omega_{C^2,p}$. These bidifferentials  form a 
locally free  sheaf  $\Omega^{(2)}_C$ of $\Ocal_{C^2}$-modules of rank one. The bidifferentials  which have a pole of order 2 along the diagonal embedding $\Delta_C: C\hookrightarrow C^2$ have there  what is called a biresidue; this is an element of  $\Ocal_C$. 
The component exchange $\sigma: C^2\to C^2$ also acts on $\Omega^{(2)}_C$ and preserves the biresidue. 
The $\sigma$-invariant   sections 
of $\Omega^{(2)}_C(2\Delta_C)$ then make up an extension of $k$ by $\sym^2H^0(C, \Omega_C)$, so that we have an exact sequence
\[
0\to \sym^2H^0(C, \Omega_C)\to H^0(C^2, \Omega^{(2)}_C(2\Delta_C))^\sigma\xrightarrow{\Bir} k\to 0. 
\]
The preimage of $1\in k$ is evidently an  affine space  for $\sym^2H^0(C, \Omega_C)$ and  Theorem \ref{thm:lagproj} produces 
an isomorphism of this with the space of Lagrangian supplements of $H^0(C, \Omega_C)$ in $H^1_{dR}(C)$.

We arrive at this identification via a local version, which concerns a  complete discrete valuation ring  $\Ocal$ with residue field $k$
(think of the $\mfrak_{C,p}$-adic completion of $\Ocal_{C,p}$ for some $p\in C$).  Let  $d\colon  \Ocal\to \Omega$ stand for  the universal  continuous $k$-derivation and extend this to  its field of fractions $K$,   giving  $d\colon K\to \Omega_K$. Then we have a residue pairing 
$(f,g)\in K\times K\mapsto \res(fdg)$  which is continuous, alternating and has kernel $k$ (hence becomes nondegenerate on the the topological $k$-vector space  $K/k$). Note that $\Ocal/k\subset K/k$ is a closed Lagrangian subspace for this pairing. There is an  analogue $\Omega^{(2)}$ for the module  bidifferentials for  which we have similarly defined extension 
\begin{equation}\label{eqn:bidiff}
0\to (\Omega^{(2)})^\sigma\to (\Omega^{(2)})^\sigma\xrightarrow{\Bir}k\to 0.
\end{equation}
Proposition \ref{prop:Lagrangiansupp} amounts to the assertion  that the preimage of $1$ parametrizes the Lagrangian  supplements of 
$\Ocal/k$ in $K/k$. 

In the  next  two sections we take $k=\CC$ and switch from a de Rham  setting to a Betti setting (with its ensuing Tate twists). A  bidifferential on $C^2$ can be made into a $2$-form by replacing in the description above $dt_1dt_2$ by $dt_1\wedge dt_2$. This turns a $\sigma$-invariant  bidifferential  into a $\sigma$-anti-invariant $2$-form.
We show that then the Lagrangian supplement of $H^{1,0}(C)$ in $H^2(C; \CC)$ defined by $H^{0,1}(C)$ is precisely given by the 
$\sigma$-anti-invariant $2$-form found by Biswas-Colombo-Frediani-Pirola \cite{BCFP}. 

Despite the fact that is has an order 2 pole along the diagonal, this 2-form  defines in fact a cohomology class on $C^2$. 
We express this class  in terms of the Hodge decomposition of $H^1(C; \CC)$ (Theorem \ref{thm:uchar}) and derive some of its properties. In section \ref{sect:higherpowers} we describe an interesting $n$-dimensional generalization of this from on $C^n$, which plays a role in a WZW-model (see our earlier posting \cite{looij}, where some of this is already established), but may have  an independent interest.  It defines a cohomology class on the configuration space $\conf_n(C)$, which, as for $n=2$,  depends on the complex structure but, we have not been able to describe that class in the same spirit as for  $n=2$.

In the last section we return to the setting of Section \ref{sect:lagproj}. A modification of the above sequence 
\eqref{eqn:bidiff} takes the form

\begin{equation}\label{eqn:omegaseq}
0\to \Omega_K^{\otimes 2}\to \hat\Omega_K^{\otimes 2}\xrightarrow{\Bir} k\to 0.
\end{equation}
for which  $(\Omega^{(2)})^\sigma$  is replaced by an extension by
its  reduction to the diagonal. The topological $k$-dual of that sequence is 
\begin{equation}\label{eqn:thetaseq}
0\to k\to \hat \theta_K \to \theta_K\to 0, 
\end{equation}
where $\theta_K$ is the vector space of continuous $k$-derivations $K\to K$ (the pairing of $ \Omega_K^{\otimes 2}$ with $\theta_K$ is given by contaction, thus  yielding an element of $\Omega_K$, followed by the residue map). 
The module $\theta_K$ is a topological Lie algebra 
and we show (using a Fock model) that the above exact sequence is in a canonical way a short exact sequence of Lie algebra's.  We identify this with the Virasoro extension. The sequence \eqref{eqn:omegaseq} does not come as a split sequence and hence neither does \eqref{eqn:thetaseq}. We therefore believe that this procedure  gives the most natural   construction of  this extension.

\section{Bidifferentials and Lagrangian projections}\label{sect:lagproj}
In this section, $k$ is a field of characteristic  zero. As of Subsection \ref{subsection:projstr}, this field is algebraically closed and  $C$ is a connected projective smooth curve over $k$ whose genus we denote by $g(C)$.  

\subsection{Biresidues and projective structures}\label{subsect:bires}
Let $\Ocal$ be  a complete discrete valuation ring whose residue field $k$ is of characteristic  zero.
We write $\mfrak\subset \Ocal$ for its maximal ideal and $K$ for its field of fractions $K$. 
We  regard $\Ocal$ and associated  $\Ocal$-modules (such as $K$) as endowed with the $\mfrak$-adic topology. 
The example to keep in mind  is the formal completion of a local ring of a smooth curve.
We denote by   $\theta$ the $\Ocal$-module of continuous $k$-derivations $\Ocal\to \Ocal$ and by
$d: \Ocal\to \Omega$ the universal continuous $k$-derivation. Both  $\theta$ and $\Omega$ are $\Ocal$-modules of rank one and each others $\Ocal$-dual. We use $K$ as a subscript when we tensor up over $\Ocal$ with $K$. For example,  $\Omega_K:=K\otimes_{\Ocal} \Omega$. A uniformizer $t\in \mfrak$ identifies  $\Ocal$ with $k[[t]]$, $K$ with $k((t))$,  $\Omega$ with $\Ocal dt$ and   $\theta$ with $\Ocal \frac{d}{dt}$. 

We may sometimes describe a concept in terms of a uniformizer (always denoted $t$),  but it is then understood  that the notion is independent of this choice. For example, there is a naturally defined residue map $\res\colon\Omega_K\to k$, which assigns to $\alpha=\sum_i a_i t^{i-1}dt\in \Omega_K$ the coefficient $a_0$. The residue pairing 
\[
K\times \Omega_K \to k,\quad  (f,\alpha)\mapsto \res (f\alpha)
\]
is a perfect pairing of topological  $k$-vector spaces. It gives rise to an  antisymmetric $k$-bilinear pairing 
\[
(f, g)\in K\times K\mapsto \la f | g\ra :=\res (f\, dg)\in k
\]
whose  kernel is $k\subset K$. For example, $\la t^p| t^q\ra= q\delta_{p+q,0}$. We regard the quotient $k$-vector space $K':=K/k$ endowed with the  residue pairing as a topological symplectic  space. The derivation $d$  maps $K'$ isomorphically onto the subspace $\Omega'_K\subset \Omega_K$ of differentials with zero residue. The space $K'$  contains the  image $\mfrak'$ of  $\mfrak$ in $K'$ (this is also the image of $\Ocal$) as a Lagrangian subspace; note that this image   is   isomorphically mapped by $d$  onto $\Omega$. 

\begin{remark}\label{rem:symplecticlie}
For every  $D\in \theta_K$, $f,g\in K$, we have $\res (Df.dg)+\res (Dg. df)=0$, for 
\[
Df.dg +f.dDg=Df.dg -df.Dg +d(f.Dg)=d(f.Dg)
\] 
(the first two terms cancel: if $D=u\frac{d}{dt}$, then $Df.dg =u f'g'dt=df.Dg$). In other words, $D$ infinitesimally preserves the simplectic form on $K'$.  This suggests that we may regard  $\theta_K$ as the (topological) Lie algebra $\splie (K')$. We shall however find that  for this interpretation to be useful, it needs to be modified.
\end{remark}

\smallskip
We denote the $\mfrak\otimes_k\mfrak$-adic completion of $\Ocal\otimes_k \Ocal$ resp.\  
$K\otimes_k K$ by  $\Ocal^{(2)}$ resp.\  $K^{(2)}$. So these  consist of the formal power series 
$\sum_{i,j} a_{k_1, k_2} t_1^{k_1}t_2^{k_2}$ with $a_{k_1, k_2}=0$ when $\min{\{k_1, k_2\}}$ is smaller 
than zero resp.\ some integer.  The diagonal defines a restriction map $\Delta^*\colon \Ocal^{(2)}\to \Ocal$ which takes 
$f_1\otimes f_2$ to $f_1f_2$. Its kernel, which we denote by $\Ical_\Delta\subset \Ocal^{(2)}$,  is the ideal generated by $t_1-t_2$ and hence principal.

The $\Ocal^{(2)}$-module  $\Omega^{(2)}:=\Omega\hat\otimes_k \Omega$ resp. 
$K^{(2)}$-module $\Omega_{K^{(2)}}^{(2)}:=\Omega_K\hat\otimes_k \Omega_K$ of \emph{bidifferentials} is free of rank one and generated by $dt_1dt_2$. We here regard $dt_1dt_2$ as a symmetric tensor,  so that it is invariant under the  transposition  $\sigma$ which exchanges the two factors:  $dt_1dt_2=dt_2dt_1$. This makes $\sigma$ act on $\Omega_{K^{(2)}}^{(2)}$.

Let  $\Omega^{\otimes 2}:=\Omega\otimes_\Ocal \Omega$ stand for the $\Ocal$-module of quadratic differentials 
(so this has $(dt)^2$ as a generator). Restriction  to the diagonal defines a map 
$\Delta^*\colon \Omega^{(2)}\to \Omega^{\otimes 2}$ (which takes  $dt_1d t_2$ to $(dt)^2$) whose kernel is $\Ical_\Delta\Omega^{(2)}$. 
There is a canonically defined  `biresidue map'  
\[
\Bir\colon \Ical_\Delta^{-\infty}\Omega_{K^{(2)}}^{(2)}\to K, 
\]
where $\Ical_\Delta^{-\infty}$ stands for $\cup_{n>0} \Ical_\Delta^{-n}$.  In terms  of our uniformizer: if we expand   $\eta\in \Ical_\Delta^{-\infty}\Omega_{K^{(2)}}^{(2)}$ as a series 
$\sum_{n\ge n_0} (t_1-t_2)^{n}f_n(t_2)dt_1dt_2$ (with $f_n\in K$), then $\Bir(\eta)=f_{-2}(t)$.  
We could here have exchanged  the roles of $t_1$ and $t_2$: the  biresidue  is $\sigma$-invariant.

We also have the  somewhat more conventional  looking residue map 
\[
\res_{1}\colon    \Ical_\Delta^{-\infty}\Omega_{K^{(2)}}^{(2)}\to K
\]
 which assumes $t_1$ to be small compared with $t_2$: it  is obtained by expanding  
$(t_1-t_2)^{-1}$ as a series $t_2^{-1}(t_1/t_2-1)^{-1}=-\sum_{i\ge 0} t_1^it_2^{-1-i}$  and then 
take the ordinary residue in zero with respect to the $t_1$-variable.
The residue  $\res_{2}\colon   \Ical_\Delta^{-\infty}\Omega_{K^{(2)}}^{(2)}\to K$ is similarly defined (we then expand 
$(t_1-t_2)^{-1}$ as $t_1^{-1}(1- t_2/t_1)^{-1}=\sum_{i\ge 0} t_1^{-1-i}t_2^i$).
We also have a residue for $\Ical_\Delta^{-2}\Omega_{K^{(2)}}^{(2)}$ along the diagonal, but here the ordering matters because $\Omega_{K^{(2)}}^{(2)}$ consists of bidifferentials, not of $2$-forms. We define
\[
\res_{1\to 2}: \Ical_\Delta^{-\infty}\Omega_{K^{(2)}}^{(2)}\to \Omega_K,
\]
the notation suggesting  that the second variable is fixed and it the first one that moves. So if we write 
$\eta \in \Ical_\Delta^{-\infty}\Omega_{K^{(2)}}^{(2)}$ as $\eta_1dt_2$, where $\eta_1\in \Ical_\Delta^{-\infty}K^{(2)}dt_1$  (we just split off $dt_2$), then $\res_{1\to 2}(\eta)$ is obtained by taking the standard residue of $\eta_1$ along $\Delta$ and multiply the result  with $dt_2$. In terms of the above expansion, $\res_{1\to 2}(\eta)=f_1(t)dt$. 
Note that  $\res_{2\to 1}(\eta)=-\res_{1\to 2}(\eta)$.

The biresidue  restricted to  $\Ical_\Delta^{-2}\Omega^{(2)}$ takes values in  $\Ocal$
and has kernel $\Ical_\Delta^{-1}\Omega^{(2)}$. Consider the quotient $\Ical_\Delta^{-2}\Omega^{(2)}/\Omega^{(2)}$.
It inherits an action of $\sigma$. Its $\sigma$-anti-invariant part is $\Ical_\Delta^{-1}\Omega^{(2)}/\Omega^{(2)}$, which is also the kernel of the biresidue map $\Ical_\Delta^{-2}\Omega^{(2)}/\Omega^{(2)}\to \Ocal$. 
We denote by $\hat\Omega^{\otimes 2}$ the $\sigma$-invariant part of $\Ical_\Delta^{-2}\Omega^{(2)}/\Omega^{(2)}$ on which the biresidue is constant, i.e., takes  its values in $k$. So we have  an exact sequence of $k$-vector spaces 
\begin{equation}\label{eqn:biresseq1}
0\to \Omega^{\otimes 2}\to \hat\Omega^{\otimes 2}\xrightarrow{\Bir} k\to 0.
\end{equation}

\begin{remark}\label{rem:conformal}
If $\eta\in \hat\Omega^{\otimes 2}$ has biresidue constant equal to $1$ (and so splits the above sequence), then  
it has the form $(t_1-t_2)^{-2}dt_1dt_2 +f(\tau) dt_1dt_2 $ ($\tau$ uniformizes the diagonal: we can substitute $t_1$ or $t_2$ for it). It is 
well known that there is uniformizer $t$ of $\Ocal$ such that the second term becomes zero and that  such a $t$ is unique up to a fractional linear transformation:  any other uniformizer $t'$ with this property is of the 
form $t/(ct+d)$ with $c\in k$ and $d\in k^\times$ (for an arbitrary uniformizer the second term is $\tfrac{1}{6}$ times the Schwarzian derivative).  So to give $\eta\in\hat\Omega^{\otimes 2}_\Ocal$ is 
equivalent to giving  $\Ocal$ a projective structure (that is, a $\PGL_2(k)$-orbit of isomorphisms of $\Ocal$ 
with a completed local ring of $\PP^1_k$). Since $1/t'= c+d/t$, the polynomial  subalgebra of $K$
generated by $1/t$ only depends on $\eta$. This subalgebra  defines  a supplement of $\mfrak$ in $K$ and is Lagrangian for the residue pairing.   We will however be interested in another Lagrangian supplement
(which  need not be multiplicatively closed)  namely the one obtained in Proposition   \ref{prop:Lagrangiansupp} below. 
\end{remark}

We shall consider several variations on the short exact sequence \eqref{eqn:biresseq1}. The one that is relevant here is
the preimage of $k\subset K$ under the biresidue map $\Ical_\Delta^{-2}\Omega^{(2)}\to K$. We denote this preimage $\hat\Omega^{(2)}$, so that we have an exact sequence $0\to \Omega^{(2)}\to \hat\Omega^{(2)}\xrightarrow{\Bir} k\to 0$. The transposition $\sigma$  acts on $\hat\Omega^{(2)}$ nontrivially, but  its anti-invariants lie of course in $\Omega^{(2)}$. We focus on the $\sigma$-invariant part 
\begin{equation}\label{eqn:biresseq2}
0\to (\Omega^{(2)})^\sigma \to (\hat\Omega^{(2)})^\sigma\xrightarrow{\Bir} k\to 0.
\end{equation}

We recall that $K':=K/k$ and $\Omega'_K:=\ker(\res\colon\Omega_K\to k)$.


\begin{proposition}\label{prop:Lagrangiansupp}
Let $\eta\in (\hat\Omega^{(2)})^\sigma$ have biresidue constant $1$. 

Then  $\res_{1\to 2}\pi_1^*(f)\eta=df$  for every $f\in K$. Furthermore, the map
\[
S_\eta\colon K\to \Omega_K \quad f\mapsto \res_{1} \pi_1^*(f)\eta,
\]
is zero on $\Ocal$ and its  image  is a Lagrangian supplement  of $\Omega$ in $\Omega'_K$. If we use the universal derivation to identify $K'$ with $\Omega'_K$, so that $S_\eta$ factors though  an endomorphism $S'_\eta$ of $K'$, then $-S'_\eta$ is a projector with kernel $\mfrak'$ and image a Lagrangian supplement $L_\eta$  of $\mfrak'$ in $K'$.

In particular, $f\mapsto \res_{1\to 2}\pi_1^*(f)\eta+ \res_{1} \pi_1^*(f)\eta$ factors through  a Lagrangian  
projection $\Pi_\zeta$ of $K'$ onto $\mfrak'$ with kernel $L_\eta$. 

The assignment $\eta\mapsto L_\eta$  
establishes an isomorphism between the elements of $( \hat\Omega^{(2)})^\sigma$ with biresidu $1$ and the 
space of  Lagrangian supplements of $\mfrak'$ in $K'$. This is an isomorphism of $\Omega^{(2)}$-torsors.
\end{proposition}
\begin{proof}
By definition  $\eta=(t_1-t_2)^{-2}dt_1dt_2 +\eta_0$ for some symmetric
$\eta_0\in \Omega^{(2)}$. 
The first assertion follows from the fact that
\[
\res_{1\to 2}\pi_1^*(f)\eta=\res_{1\to 2}f(t_1)(t_1-t_2)^{-2}dt_1dt_2= f'(t_2)dt_2,
\]
where for the last equality we used a classical residue property. 

In order  to check the  assertions regarding $S_\eta(t^n)$ with  $n\in \ZZ$, we first compute 
\begin{multline*}
\textstyle \res_1 t_1^n(t_1-t_2)^{-2}dt_1dt_2 = \res_1t_1^{n}t_2^{-2}(t_1/t_2-1)^{-2}dt_1dt_2=\\=
\textstyle \res_1 t_1^{n}t_2^{-2}\sum_{i\ge 0} (i+1)(t_1/t_2)^i dt_1dt_2= \res_1\sum_{i\ge 0}  (i+1)t_1^{n+i}t_2^{-2-i}dt_1dt_2. 
\end{multline*}
The value of the last expression is produced by the term for which $n+i=-1$, or equivalently, $n=-(i+1)$. This occurs only when $n<0$ (otherwise the residue is zero)  and then  gives $-nt_2^{n-1}dt_2= -d(t_2^n)$.

We next  expand  $\eta_0$ as $\sum_{i\ge 1} t_1^{i-1}dt_1.\pi_2^*(df_i)$ with $f_i\in \Ocal$. Then 
\[
\textstyle \res_1 (t_1^n\eta_0) =\sum_{i\ge 1} \res_1 t_1^{i-1+n}dt_1.\pi_2^*(df_i)
\]
and this is zero unless $n<0$ and $i=-n$: then its value is $df_{-n}$.
It follows that $S_\eta$   vanishes on $\Ocal$ and takes $t^{-n}$ ($n>0$) to $d(-t^{-n}+f_n)$. In particular,
$S_\eta(t^{-n}-f_{n})= S_\eta(t^{-n})=d(-t^{-n}+f_{n})$. 

The $k$-span  $L_\eta\subset K$ of  $\{t^{-n}-f_{n}\}_{n=1}^\infty$ clearly supplements $\mfrak'$ in $K'$ and we just proved that $dL_\eta$ is the image of  $S_\eta$. One checks that the  symmetry property of $\eta_0$ implies that  $L_\eta$ is Lagrangian for the residue pairing. 

Finally, any supplement  $L$ of $\mfrak'$ in $K'$ has a unique basis of the form $\{t^{-n}-f_{n}\}_{n=1}^\infty$ with $fn\in \Ocal$. We then put
$\eta_L:=(t_1-t_2)^{-2}dt_1dt_2 +\sum_{i\ge 1} t_1^{i-1}dt_1.\pi_2^*(df_i)$. One checks that $\eta_L$ is $\sigma$-invariant  if and only if $L$ is
Lagrangian and that $L=L_{\eta_L}$.
\end{proof}

\subsection{The sheaf of projective structures}\label{subsection:projstr}
As of now,    $k$ is  algebraically closed (and  of characteristic zero) and $C$ a connected projective smooth curve over $k$ of  genus by $g(C)$.  This subsection as well as the next contain little that is new; they merely recall known material.

The involution of $C^2$ which exchanges factors will be denoted by $\sigma$. We write 
\[
\Omega_C^{(2)}:=\pi_1^*\Omega_C\otimes_{\Ocal_{C^2}}\pi_2^*\Omega_C
\]
 for the sheaf of bidifferentials on $C^2$. We regard  this however as a subsheaf of $\sym^2(\Omega^1_{C^2})$ (in terms of a pair of local coordinates $z_1, z_2$ on $C$, this amounts to identifying $dz_1dz_2$ with $dz_2dz_1$), so that $\Omega_C^{(2)}$ comes with an action of $\sigma$. 

Let $\Delta_C\colon C\to C^2$ be the diagonal embedding or just stand for its image.
It is clear that  $\Omega_C^{(2)}(2\Delta_C)/\Omega_C^{(2)}(-\Delta_C)$ is supported by $\Delta_C$ and that the  component  exchange   $\sigma$ acts on  this sheaf. The $\sigma$-anti-invariant part maps isomorphically to $\Omega_{C^2}^2(\Delta_C)/\Omega_{C^2}^2$ (taking the residue along $\Delta_C$ identifies with sheaf with $\Omega_C$)  and  the $\sigma$-invariant part is an extension 
 \[
0\to \Omega_C^{\otimes 2}\to \Big(\Omega_{C^2}^2(2\Delta_C)/\Omega_{C^2}^2(-\Delta_C)\Big)^\sigma \to \Ocal_C\to 0,
\]
where the arrow to $\Ocal_C$ is given by the biresidue. Let us write  $\hat\Omega_C^{\otimes 2}$ for the middle term of the exact sequence above,  so that have the exact sequence of abelian sheaves
\begin{equation}\label{eqn:quadraticext}
0\to \Omega_C^{\otimes 2}\to \hat\Omega_C^{\otimes 2}\to \Ocal_C\to 0,
\end{equation}
(NB: the middle term is \emph{not} a sheaf of $\Ocal_C$-modules).
The preimage of $1\in\Ocal_C$ is defines  a $\Omega_C^{\otimes 2}$-torsor, which we may interpret as the sheaf of 
 projective structure on $C$.  So the following (well-known) lemma is  also a consequence of the existence of a  projective structure on $C$.

\begin{lemma}\label{lemma:}
The section sequence  of  \eqref{eqn:quadraticext}   gives the exact sequence 
\begin{equation}\label{eqn:extensionH0}
0\to H^0(C,\Omega_C^{\otimes 2})\to H^0(C, \hat\Omega_C^{\otimes 2})\to k\to 0.
\end{equation}
\end{lemma}
\begin{proof}
When $g(C)>1$, this is immediate from the vanishing of $H^1(C,\Omega_C^{\otimes 2})$.  
In the two remaining cases, we proceed as follows.

When $g(C)=0$, the space  $H^0(C,\Omega_C^{\otimes 2})$ is zero and then its is a matter of finding a 
nonzero  element of  $H^0(C, \hat\Omega_C^{\otimes 2})\cong k$. Indeed,  in terms of an affine coordinate 
$z$, the form  $(z_1-z_2)^{-2}dz_1 dz_2$  does the job. 

When $g(C)=1$, then choose a nonzero differential 
$\alpha$ on $C$ (which is then necessarily translation invariant). Choose also $o\in C$  and a rational differential $\beta$ on $C$ whose  whose pole divisor is $2(o)$ and which is 
invariant under the involution defined by $o$ (for example $\wp\alpha$, where $\wp$ is  Weierstra\ss's function). The pull-back of $\alpha\otimes\beta$ under the map 
$(p,q)\in C\times C\mapsto (p, q-p)\in C\times C$ is then a rational $2$-form on $C^2$ whose polar divisor 
is $2\Delta_C$ and which is invariant with respect to the transposition. Its biresidue is nonzero and hence the sequence \eqref{eqn:extensionH0}is still exact.
\end{proof}

The extension \eqref{eqn:quadraticext} lifts  to a  corresponding  extension over all of $C^2$ 
\begin{equation}\label{eqn:quadraticext2}
0\to\Omega_C^{(2)}\to \hat{\Omega}_C^{(2)}\to \Delta_{C*}\Ocal_C\to 0.
\end{equation}
 We can interpret a result of   Biswas-Raina (Prop.\ 2.10 of \cite{br1}) as saying that the associated long exact cohomology sequence  splits up in two short exact  sequences, the first of which is 
\begin{equation}\label{eqn:quadraticext3}
0\to  H^0(C, \Omega_C)^{\otimes 2}\to H^0(C^2,\hat{\Omega}_C^{(2)})\to k\to 0
\end{equation}
with $\sigma$ acting on $H^0(C, \Omega_C)^{\otimes 2}$ by exchanging tensor factors, but acting trivially on $k$.

\begin{question}\label{rem:}
The second exact sequence can be written as 
\[
0\to H^0(C,\Omega_C)\oplus H^0(C,\Omega_C)\to H^1(C^2,\hat{\Omega}_C^{(2)})\to k\to 0,
\] 
where $\sigma$ acts on $H^0(C,\Omega_C)\oplus H^0(C,\Omega_C)$ by exchanging summands (and acting trivially on $k$). Although we shall not use that sequence, we mention it, because its $\sigma$-invariant part produces a canonical extension of $k$ by $H^0(C,\Omega_C)$  (or dually, an extension of $H^1(C, \Ocal_C)$ by $k$). We wonder whether or not that extension is canonically split (we suspect it is not).
\end{question}

\subsection{Brief review of the Rham cohomology of $C$}\label{subsect:dR}  We first fix some notation. For $p\in C$, we denote the   local ring $\Ocal_{C,p}$ 
completed with respect to its maximal idea by $\Ocal_p$. This  is a complete discrete valuation ring as considered at the beginning of this section. We adopt a notation that agrees  with  the notation introduced there, so that  
 $\mfrak_p\subset \Ocal_p\subset K_p$, $\theta_p\subset \theta_{K_p}$, and $\Omega_p\subset \Omega_{K_p}$ have the obvious meaning. We extend this to the case of a finite subset  $P\subset C$. So 
$\Ocal_P:=\prod_{p\in P} \Ocal_p$ and likewise for the other items. We sometimes regard these as (sections of) sheaves over $P$.

We recall the (adelic) description of the first cohomology of a  quasi-coherent sheaf  $\Fcal$ of  $\Ocal_C$-modules (as for example explained in Ch.~II of Serre's book \cite{serre}).
Let  $P\subset C$ be a finite nonempty subset. Then  an `affine covering' of $C$  consists of  the  affine curve $C\ssm P\subset C$  and a formal neighborhood of $P$. To be precise, if $\Fcal$ is a coherent $\Ocal_C$-module,  then $H^1(C, \Fcal)$ can be obtained in these terms as follows.  Denote by $\Fcal[C\ssm P]$ the group of sections of $\Fcal$ over $C\ssm P$. It is clear that $\Ocal_C[C\ssm P]$ is the $k$-algebra $k[C\ssm P]$ defining the affine curve 
$C\ssm P$ and so $\Fcal[C\ssm P]$ is a $k[C\ssm P]$-module. We have a natural map
 $\Fcal [C\ssm P] \to K_P\otimes_{\Ocal_{C, P}}i_P^{-1}\Fcal$, where $i_P: P\subset C$. If we denote the cokernel of its composite with $K_P\otimes_{\Ocal_{C, P}}i_P^{-1}\Fcal\to (K_P/\Ocal_P)\otimes_{\Ocal_{C, P}}i_P^{-1}\Fcal$ provisionally by $C(P)$, then
 it is easy to check that for a nonempty  finite $Q\supset P$ we have a natural $k$-linear isomorphism  
 $C(P)\xrightarrow{\cong} C(Q)$ .
 So $C(P)$  maps isomorphically the inductive limit $\varinjlim_{Q} C_Q$ and has therefore an intrinsic meaning. Indeed, this is the adelic description of  $H^1(C, \Fcal)$  by means of `repartitions'. We have  in particular an exact sequence 
 \[
\Fcal [C\ssm P] \to  (K_P/\Ocal_P)\otimes_{\Ocal_{C, P}}i_P^{-1}\Fcal \to    H^1(C, \Fcal)\to 0. 
 \]
If $\Fcal$ is locally free, the maps  from  $i_P^{-1}\Fcal$  and 
$\Fcal[C\ssm P]$ to  $K_P\otimes_{\Ocal_{C, P}}i_P^{-1}\Fcal$ are both injective, so if we regard these as inclusions, then 
 \[
 H^1(C, \Fcal)\cong K_P\otimes_{\Ocal_{C, P}}i_P^{-1}\Fcal/\big(i_P^{-1}\Fcal + \Fcal[C\ssm P]\big).
 \]

For example,  with the notional conventions above, 
$H^1(C, \Omega_C)\cong \Omega_{K_P}/(\Omega_P+\Omega_C[C\ssm P])$. It is a basic fact 
(and a  consequence of Riemann-Roch) that  a polar part of a $1$-form at $P$, i.e., an element  of 
$\Omega_{K_P}/\Omega_P$,  is realized by an element  of $\Omega_C[C\ssm P]$ if and only if  
that polar part has zero  residue sum. This means that  taking the residue sum  gives a well-defined isomorphism 
\[
\Tr_C: H^1(C, \Omega_C)\cong \Omega_{K_P}/(\Omega_P+\Omega_C[C\ssm P])\xrightarrow{\sum_{p\in P} \res_p} k. 
\]
This is  also called the \emph{trace map} (whence the notation). It is the de Rham  analogue of integration, but the value of the latter differs by a factor $2\pi\sqrt{-1}$.

Consider the following filtration of $K_P$:
\begin{equation}\label{eqn:basicfiltration}
0\subset k[C\ssm P]\subset d^{-1}\Omega_C[C]+ k[C\ssm P]\subset d^{-1}\Omega_C[C\ssm P]\subset K_P.
\end{equation}
Here $\Omega_C[C]=H^0(C, \Omega_C)$ and so $d^{-1}\Omega_C[C]$ consists of the $f\in \Ocal_p$ for which $df$ is
the restriction of a regular $1$-form. Hence the first subquotient $(d^{-1}\Omega_C[C]+ k[C\ssm P])/k[C\ssm P]$maps isomorphically onto $H^0(C, \Omega_C)$. For the second subquotient, we note that 
$d^{-1}\Omega_C[C\ssm P]\cap \Ocal_P=d^{-1}\Omega_C[C]$ and hence 
\[
\frac{d^{-1}\Omega_C[C\ssm P]}{d^{-1}\Omega_C[C]+ k[C\ssm P]}\cong
\frac{d^{-1}\Omega_C[C\ssm P]+\Ocal_P}{k[C\ssm P]+\Ocal_P}\cong \frac{K_P}{k[C\ssm P]+\Ocal_P}\cong H^1(C, \Ocal_C).
\]
By the preceding 
$d^{-1}\Omega_C[C\ssm P]/k[C\ssm P]$ may be identified  with the cokernel of  
$d: k[C\ssm P] \to \Omega_C'[C\ssm P]$, where $\Omega_C'[C\ssm P]\subset \Omega_C[C\ssm P]$ is space of $1$-forms on $C\ssm P$ having zero residue sum. This cokernel is the first De Rham cohomology space $H^1_{dR}(C)$ of $C$.  We thus obtain the familiar short exact sequence
\begin{equation}\label{eqn:dRclassical}
0\to H^0(C, \Omega_C)\to H^1_{dR}(C)\to H^1(C, \Ocal_C)\to 0. 
\end{equation}
The residue pairing 
\[
\textstyle (f, g)\in K_P\times K_P\mapsto \la f|g\ra:=\sum_{p\in P} \res_p(fdg)\in k
\]
has the property that of the filtration \eqref{eqn:basicfiltration} of $K_P$,  the subspaces $k[C\ssm P]$ and $ d^{-1}\Omega_C[C\ssm P]$ are each others annihilator. So it induces a perfect  pairing
\begin{equation}\label{eqn:intersectionp}
\la\; |\; \ra \colon H^1_{dR}(C)\times H^1_{dR}(C)\to k.
\end{equation}
This is the  intersection pairing for de Rham cohomology,  which in case $k=\CC$ differs  from its topological analogue by a factor $2\pi\sqrt{-1}$. The middle term of the filtration $d^{-1}\Omega_C[C]+k[C\ssm P]$ is its own annihilator and so 
$H^0(C, \Omega_C)\hookrightarrow H^1_{dR}(C)$ has a Lagrangian image. This  implies that
we have a perfect pairing $H^0(C, \Omega_C)\times H^1(C, \Ocal_C)\to k$. This is of course just the Serre duality pairing; it can be obtained as the composite of  taking the cup product $H^0(C, \Omega_C)\times H^1(C, \Ocal_C)\to H^1(C, \Omega_C)$,  followed by the  trace map 
$\Tr_C: H^1(C, \Omega_C)\to k$. 

\subsection{Lagrangian projections in $H^1_{dR}(C)$}\label{subsect:retract}
We shall use a relative version of the above discussion, namely one which concerns the projection $\pi_2:C\times C\to C$ onto the second factor. 

In what follows we fix a symmetric   $\eta\in H^0(C^2,\hat{\Omega}_C^{(2)})$ with  biresidue $1$.
There is an associated  rational $2$-form on $C^2$ (with polar divisor $2\Delta_C$) that we shall denote by $\zeta$: if $(p, q)\in C^2$ and we have local coordinates $z_1$ at $p$ and $z_2$ at $q$, then we replace $dz_1dz_2$ by $dz_1\wedge dz_2$. So $\zeta$ is now anti-invariant:
$\sigma^*\zeta=-\zeta$. We shall see that $ \zeta$ represents a class in $ H^2_{dR}(C^2)$, so that is makes sense to consider 
the  $k$-linear endomorphism
\begin{equation}\label{eqn:Tcorresp}
\alpha\in H^\pt_{dR}(C)\mapsto  \pi_{2*}(\pi_1^*(\alpha)\cup  \zeta)\in H^\pt_{dR}(C). 
\end{equation}
We will however work with $\eta$ and as we shall see in Section \ref{sect:canform}, this makes quite  a difference.

We here focus on the degree 1 part (which is the most interesting anyway). In that case 
$\pi_{2*}$ is the `fiberwise de Rham intersection product' along $\pi_2$. 
We fix some  $p\in C$,  so that we can identify $H^1_{dR}(C)$ with $d^{-1}\Omega_C[C\ssm\{p\}]/ k[C\ssm\{p\}]$.
We then  show how  the  above edomorphism  can be understood on  the form level as inducing a  $k$-linear map $K'_p\to K'_p$ which preserves  $d^{-1}\Omega[C\ssm\{p\}]$ and $k[C\ssm\{p\}]$ (thereby inducing an endomorphism of $ H^1_{dR}(C)$).

We define $\Pi^{dR}_\eta(\alpha)\in \Omega[C\ssm\{p\}]$ for  $\alpha\in\Omega[C\ssm\{p\}]$ as follows.  Given $q\in C$, we choose a formal local parameter $z$ at $q$ and then write 
the coherent restriction of $\eta$ to $C\times \{q\}$ as $\eta(q)dz_2(q)$, where $\eta(q)$ is a   
$1$-form  on $C$ with a pole  of order 2 at $q$ and $ dz(q)$ denotes the value of $dz$ in $q$ (so this is an element of the cotangent space of $C$ at $q$) and stipulate that the value of 
$\Pi^{dR}_\eta(\alpha) $ in $q$ is  $\la \alpha |\eta(q)\ra dz(q)$.
In order to compute $\la \alpha |\eta(q)\ra$, let us first assume $p\not=q$. The residue formula shows that the intersection product  has two contributions coming from $p$ and $q$: the image of $\alpha$ in $\Omega_q$  (the formal completion of $\Omega_{C,q}$) can be written as $df_q$  for some $f\in \Ocal_q$  and  since $\alpha$ has zero residue at $p$  (by the residue theorem, for $\alpha$ has no other poles) this  is also true at $p$, except that we must take $f_p\in K_p$. So then
\[
\Pi^{dR}_\eta(\alpha)(q)= \big(\res_q f_q\eta(q) +\res_p f_p\eta(q)\big)dz(q)
\]
If $\alpha$ is exact, so that  $\alpha=df$ for some $f\in k[C\ssm \{p\}]$, then by the residue theorem applied to $f\eta(q)$, the two residues have  zero sum and hence $\Pi^{dR}_\eta(\alpha)(q)=0$.
This description makes it clear  that by regarding  $q\in C\ssm \{p\}$ as a variable, $\Pi^{dR}_\eta$ takes  $\alpha$ to an element of $\Omega_C[C\ssm \{p\}]$.  Thus $\Pi^{dR} _\eta$ preserves $\Omega[C\ssm \{p\}]$ and  kills the exact forms. It therefore  induces an endomorphism of  $H^1_{dR}(C)$ (that we shall also denote by $\Pi^{dR}_\eta$).

When $p=q$, we can do this on a formal neighbourhood of $(p,p)$. This shows that 
$\Pi^{dR}_\eta(\alpha)$ is at $p$ given by
\[
\res_{1\to 2} f_p\eta +\res_1 f_p\eta,
\]
which by Proposition  \ref{prop:Lagrangiansupp} defines a Lagrangian   projection  $\Pi_\eta$ of  $K'_p$ onto $\mfrak'$. 

\begin{theorem}\label{thm:lagproj}
The map $\Pi^{dR}_\eta$ induces in  $H^1_{dR}(C)$ a Lagrangian projection $\Pi_\eta^{dR}$ onto  
$H^0(C, \Omega_C)$. Furthermore,  $\eta\mapsto \Pi^{dR}_\eta$ defines  an isomorphism of  
$\sym^2H^0(C, \Omega_C)$-torsors, namely space  of symmetric $\eta\in H^0(C^2,\hat{\Omega}_C^{(2)})$ 
with biresidue one  and the space of Lagrangian projections onto $H^0(C, \Omega_C)$ (or equivalently, 
the space of  Lagrangian supplements of $H^0(C, \Omega_C)$ in $H^1_{dR}(C)$).
\end{theorem}
\begin{proof}
Since $d\mfrak'\cap d^{-1}\Omega_C[C\ssm \{p\}]=H^0(C, \Omega_C)$, it  follows that $T_\eta$  is a 
Lagrangian projection  onto  $H^0(C, \Omega_C)$.

It is clear from the definition  that  the symmetric $\eta\in H^0(C^2,\hat{\Omega}_C^{(2)})$ with biresidue 1 form 
a $\sym^2H^0(C, \Omega_C)$-torsor. The Lagrangian projections onto $H^0(C, \Omega_C)$ also form one. To see this, note that we have defined a monomorphism of algebraic groups 
from the (commutative) vector group $\sym^2H^0(C, \Omega_C)$ to the symplectic group $\Sp(H^1_{dR}(C))$ by 
assigning  to $\alpha\otimes\alpha\in\sym^2H^0(C, \Omega_C)$ the map 
$\xi\in H^1_{dR}(C)\mapsto \xi+ \alpha \la \alpha|\xi\ra$. This subgroup fixes $H^0(C, \Omega_C)$ pointwise and 
makes the space of Lagrangian projections onto 
$H^0(C, \Omega_C)$ a $\sym^2H^0(C, \Omega_C)$-torsor. 
It is easy to see that $\eta\mapsto \Pi^{dR}_\eta$ is $\sym^2H^0(C, \Omega_C)$-equivariant.
\end{proof}

We shall denote the  kernel of the Lagrangian  projection $\Pi^{dR}_\eta$ by  $H^1(C, \Ocal_C)_\eta$.
So this is a  Lagrangian supplement of $H^0(C,\Omega_C)$ in $H^1_{dR}(C)$.

\section{The class of the canonical 2-form}\label{sect:canform}  
 In this section and the next   we take the complex field as base field. We adhere to the conventions of Hodge theory, which has a Betti component  and a de Rham component. In order to keep the two compatible in the sense that Betti cohomology with complex coefficients is identified with de Rham cohomology, the former will typically appear with Tate twists.

\subsection{Betti and de Rham cohomology} If  $M$ is a complex-projective manifold of complex dimension $m$, and $i\colon N\subset M$ is a closed complex submanifold of complex codimension $d$, then the Gysin map
takes the form 
\begin{equation}\label{eqn:gysinmap}
i_!\colon H^\pt(N)\to H^{2d+\pt}(M)\otimes \ZZ(d).
\end{equation}
We use the occasion to recall that  this map \eqref{eqn:gysinmap} is  $H^\pt(M)$-linear, if we consider  $H^\pt(N)$ a $H^\pt(M)$-module  via 
the  ring homomorphism $i^*\colon H^\pt(M)\to H^\pt(N)$: if $\beta\in H^\pt(N)$ and $\alpha\in H^\pt (M)$, then
\begin{equation}\label{eqn:clef}
i_!(i^*\alpha\cup\beta)=\alpha\cup i_!(\beta).
\end{equation}
The Gysin map is weight preserving, but the degree increases by $2d$.
In particular, we have defined the \emph{class}, 
\[
cl(N):=i_!(1)\in H^{2d}(M)\otimes \ZZ(d),
\]
which  represents the Hodge-Poincar\'e dual of $N$  in $M$. It  is $(2\pi\sqrt{-1})^d$ times its topological Poincar\'e dual when we use the complex orientation. So here $ \ZZ(d)$ functions as the (trivial) local system of orientations of the normal bundle of $i$.  If we take for $N$ a singleton, and $M$ is connected, then  this gives the orientation class 
\[
[M]:=i_!(1)\in H^{2m}(M)\otimes \ZZ(m). 
\]
The resulting map  $H_{2m}(M)\to\ZZ(m)$ (given by integration over $M$) is an isomorphism.
\smallskip

Let $C$ be a nonsingular connected complex-projective curve of genus $g$. As before, we  denote by $\sigma $ the involution of $C^2$ which exchanges its factors and hence  the two projections $\pi_1, \pi_2\colon C^2\to C$. So  $\sigma$  acts on the K\"unneth component $H^k(C)\otimes H^l(C)$ of $H^{k+l}(C^2)$ by 
\[
\sigma^*( \alpha\otimes\beta)=\sigma^*(\pi_1^*\alpha\cup \pi_2^*\beta)=\pi_2^*\alpha\cup \pi_1^*\beta= (-1)^{kl}\pi_1^*\beta\cup\pi_2^*\alpha=(-1)^{kl}\beta\otimes \alpha.
\]
It follows that  the group of 
$\sigma$-invariants in $H^2(C^2)$ is the span of the \emph{antisymmetric} tensors in $H^1(C)\otimes H^1(C)$ and 
$[C]\otimes 1+1\otimes [C]$, whereas  the group of 
\emph{$\sigma$-anti-invariants} in $H^2(C^2)$ (the subgroup of $H^2(C^2)$ on which $\sigma$ as multiplication by $(-1)$), is the span of the \emph{symmetric} tensors $\sym^2H^1(C)\subset H^1(C)\otimes H^1(C)$ and $[C]\otimes 1-1\otimes [C]$.

The intersection pairing on $H_1(C)$ should be regarded as a bilinear map $H_1(C)\times H_1(C)\to \ZZ(1)$. 
It therefore  defines a tensor $\delta\in H^1(C)\otimes H^1(C)\otimes \ZZ(-1)$. Concretely,  if $(\alpha_{\pm 1},\dots, \alpha _{\pm g})$ is a basis of $H^1(C)$ such that $\alpha_i\cup\alpha_{-j}=\sign(i)\delta_{ij}.[C]$, then
\begin{equation}\label{eqn:delta}
\textstyle \delta=\sum_{i=1}^g  (-\alpha_i\otimes \alpha_{-i}+\alpha_{-i}\otimes\alpha_i).
\end{equation}
So  $\sigma^*$  fixes  $\delta$.  In fact, 
\begin{equation}\label{eqn:clef0}
[\Delta_C]=[C]\otimes 1+\delta+ 1\otimes [C].
\end{equation}
It follows from the identity \eqref{eqn:clef} that  
$\Delta_!([C])=[C]\otimes [C]$ and that for $\alpha\in H^1(C)$,
\begin{equation}\label{eqn:clef1}
\Delta_!(\alpha)=\alpha\otimes [C]+[C]\otimes\alpha.
\end{equation}
It is a little exercise to check that in $H^2(C^3)$ we have the following identity
\begin{equation}\label{eqn:deltacup}
\pi_{1,2}^*\delta \cup \pi_{2,3}^*\delta =\pi_{1,3}^*\delta\cup \pi_2^*[C], 
\end{equation}
where the subscripts of $\pi$ indicate the projections on the corresponding factors.

\subsection{The canonical $2$-form}
The configuration space  $\conf_2(C)$ is the complement of the image of  the diagonal embedding $\Delta_C\colon C\to C^2$.  The map $\Delta_{C!}\colon H^1(C)(-1)\to H^3(C^2)$ is via duality identified with $\Delta_{C*}\colon H_1(C)\to H_1(C^2)$ which takes $a\in H_1(C)$ to $a\otimes 1+1\otimes a$. This map is clearly  injective. So the Gysin sequence for $\Delta_C$ gives  the short exact sequence 
\begin{equation}\label{eqn:simpleshortexactsequence}
0\to H^0(C)(-1)\xrightarrow{\Delta_{C!}} H^2(C^2)\to H^2(\conf_2(C))\to 0.
\end{equation}
Let $\zeta\in H^0(C^2,\Omega_C^2(2\Delta_C))$  have biresidue constant equal to one along the diagonal. Since taking the anti-invariant part under the exchange map $\sigma$ does not affect the biresidue,  we can also assume  that $\sigma^*\zeta=-\zeta$.  It is then unique up an anti-invariant  element of $H^0(C^2, \Omega_C^2)\cong H^0(C, \Omega_C)^{\otimes 2}$, i.e., an element of
$\sym^2H^0(C,\Omega_C)$ (which has Hodge type $(2,0)$). A priori, it defines an anti-invariant class in $H^2(\conf_2(C); \CC)$, but the short exact sequence \eqref{eqn:simpleshortexactsequence}  shows that the restriction map $H^2(C^2;\CC)\to H^2(\conf_2(C); \CC)$ induces an isomorphism on their  anti-invariant parts, and hence defines a class in 
$H^2(C^2;\CC)$. Biswas-Colombo-Frediani-Pirola \cite{BCFP} prove that there exists a unique  anti-invariant  $\zeta_C\in H^0(C^2,\Omega_C^2(2\Delta_C))$ of pure Hodge type  $(1,1)$  (their Thm.\ 5.4).  
The theorem below  identifies this class,  as well as  the endomorphism of  $H^\pt(C)$ it defines via the isomorphism 
\begin{gather*}
\Tcal: H^2(C^2)\xrightarrow{\cong} \Hom (H^\pt(C),H^\pt(C)(-1)),\quad   \zeta\mapsto \Tcal_\zeta, \text{ with}\\
\Tcal_\zeta: \alpha\in H^\pt(C)\mapsto \pi_{2*} (\pi_1^*\alpha\cup\xi) \in H^\pt(C)(-1). 
\end{gather*}
We reprove the characterization  of $\zeta_C$ in  \cite{BCFP}  along the way.  

\begin{theorem}\label{thm:uchar}
There is a unique section $\zeta_C$  of  $\Omega^2_{C^2}(2\Delta_C)$ which is anti-invariant under the exchange map $\sigma$,  whose  double  residue  along the diagonal is $1$. Its  class  lies in $H^2(C; \RR(1))$,  is of pure Hodge type $(1,1)$, and is represented by  
\[
\textstyle - \half[C]\otimes 1+1\otimes\half[C]  +\sum_{i=1}^g (\omega_i\otimes \bar\omega_i+\bar\omega_i\otimes \omega_i),
\]
where   $\omega_1,\dots , \omega_g$  is  a basis of  $H^0(C, \Omega_C)$ with the property that $\int_C \omega_i\wedge \overline\omega_j=2\pi\sqrt{-1}\delta_{ij}$. 

The associated endomorphism $\Tcal_{\zeta_C}$ of $H^\pt(C;\CC)$  is semisimple and its eigenspaces define the Hodge decomposition:  it  is on $H^{k,l}(C)$ equal to multiplication with
$(3k-l-1)\pi\sqrt{-1}$. 
\end{theorem}

\begin{remark}\label{rem:bidifftoform}
There is a simple way to pass from a symmetric bidifferential as considered in  Section \ref{sect:lagproj} to  an antisymmetric $2$-form on $C^2$, namely by replacing $dz_1dz_2$ by $dz_1\wedge dz_2$. Theorem \ref{thm:uchar} shows that the  consequences of this modest substitution reach farther than one might think: a projection becomes a semisimple automorphism.
\end{remark}

\begin{remark}\label{rem:projectivestructure}
The form  $\zeta_C$ determines a projective structure at $x$, which is clearly invariant under the automorphism group of $C$. 
This group acts transitively when $C$ has genus 
$0$ or $1$, and so this projective structure must then be the standard one.  
When the genus $g$ of $C$ is $>1$, then the universal cover of $C$ is realized by the upper half plane $\HH$   
with covering group contained in $\PSL_2(\RR)$. So this defines another  projective structure on $C$. 
But Biswas-Colombo-Frediani-Pirola \cite{BCFP} show that these two projective structures differ in general for higher genus.
\end{remark}

The short exact sequence \eqref{eqn:simpleshortexactsequence}  shows that the restriction map $H^2(C^2;\CC)\to H^2(\conf_2(C); \CC)$ induces an isomorphism on their  anti-invariant parts, which  is therefore  the direct sum of  the span of 
$[C]\otimes 1-1\otimes[C]$ and $\sym^2 H^1(C; \CC)$. 
In order to find the coefficient of $[\zeta]$ on the latter we proceed as follows.
Choose $p\in C$ and consider 
\[
Z:=C\times\{p\}-\{ p\}\times C
\]
 as an algebraic cycle on $C^2$. 
It is clear that with respect to the K\"unneth decomposition  of $H^2(C)$,  
integration over $Z$ is zero on $H^1(C)\otimes H^1(C)$ and takes on  $[C]\otimes 1$ resp.\  $1\otimes[C]$ the value $2\pi\sqrt{-1}$ resp.\ $-2\pi\sqrt{-1}$. So the coefficient in question is then computed by integrating $\zeta$ over any 
$2$-cycle on $\conf_2(C)$ that is homologous to $Z$ in $H_2(C^2)$. 
We construct such a cycle   by modifying $Z$ a bit near $(p,p)$. To this end we choose a holomorphic chart $(U,z)$ centered at $p$,
so  that $\zeta$  has on $U^2$ the form
\begin{equation}\label{eqn:chart}
\zeta =(z_1-z_2)^{-2} dz_1 \wedge dz_2 +\text{a form regular at  $(p,p)$.}
\end{equation}
For any  $\eps>0$  such that $U$  contains a  closed disk $D_{2\eps}$ mapping onto the closed disk of radius 
$2\eps$ in $\CC$, we embed the cylinder $[0, \eps]\times (\RR/2\pi\ZZ)$ in $U^2\ssm \Delta_U$ by 
\[
u: [0, \eps]\times (\RR/2\pi\ZZ)\to U^2\ssm \Delta_U; \quad  (s, \phi) \mapsto ( -se^{\sqrt{-1}\phi}, (\eps-s)e^{\sqrt{-1}\phi}).
\]
Considered as a $2$-chain, its  boundary is $\partial D_\eps\times \{p\}-\{p\} \times\partial D_\eps$. where  
$\partial D_\eps$ has  the standard counterclockwise orientation. The boundary of  $C\ssm D_\eps$ (for its complex orientation)  is $-\p D_\eps$ and so  if we add $u$ to the $2$-chain 
$(C\ssm D_\eps)\times\{p\} -\{p\} \times (C\ssm D_\eps)$, we obtain a $2$-cycle $Z_\eps$ in $C^2\ssm \Delta_{12}$.
It is clear that $Z_\eps$ is homologous to $Z$ in $C^2$. 

\begin{lemma}\label{lemma:integral}
The value of the cohomology class $[\zeta]$ on $Z_\eps$ equals $-2\pi\sqrt{-1}$. 
\end{lemma}
\begin{proof} The pull-back of $\zeta$ to $Z_\eps$ as a 2-form is nonzero on the cylindrical part of $Z_\eps$  only and so 
$ \int_{Z_\eps} \zeta=\int_{[0, \eps]\times [0, 2\pi]} u^*\zeta$.
The 2-form $u^*\zeta$ is up to a form smoothly depending on $\eps$ equal to
\[
\frac{d\big(-se^{\sqrt{-1}\phi}\big)\wedge d \big( (\eps-s)e^{\sqrt{-1}\phi} \big)}{ \big(\eps e^{\sqrt{-1}\phi)}\big)^2} = 
\frac{d\big(-se^{\sqrt{-1}\phi}\big)\wedge d \big(\eps e^{\sqrt{-1}\phi} \big)}{ \big(\eps e^{\sqrt{-1}\phi)}\big)^2} = 
\frac{-ds\wedge  \sqrt{-1}d\phi}{\eps}.
\]
So integration of $u^*\zeta$ over $[0, \eps]\times [0, 2\pi]$ and taking the limit for  $\eps\downarrow 0$ yields $-2\pi\sqrt{-1}$.
\end{proof}

As mentioned earlier, any  $\eta\in \Omega[C\ssm \{p\}]$ defines a class 
$[\eta]\in H^1(C\ssm \{p\}; \CC)=H^1(C; \CC)$ whose image in $H^1(C, \Ocal_C)$ is obtained as follows:
write in a punctured neighborhood of $p$ the form $\eta$ as $df$ (which is possible since the residue of 
$\eta$ at $p$ has to be zero) and then take the image of $f$ in the above quotient.
The  theorem of Stokes combined with the Cauchy residue formula implies the (well-known) identity
\begin{equation}\label{eqn:cauchy}
\int_C [\eta]\wedge\omega =2\pi\sqrt{-1}\res_p f\omega.
\end{equation}
for any holomorphic differential $\omega$ on $C$ (compare this with its de Rham version \eqref{eqn:intersectionp}).

\begin{proof}[Proof of Theorem \ref{thm:uchar}]
Let $\zeta$ be as above so that $\int_Z \zeta=-2\pi\sqrt{-1}$,  by Lemma \ref{lemma:integral}.
On the other hand $\int_Z ([C]\otimes 1-1\otimes [C])=2\pi\sqrt{-1}+2\pi\sqrt{-1}=4\pi\sqrt{-1}$ and 
hence the  coefficient of $[C]\otimes 1-1\otimes [C]$  in $[\zeta]$ is $-\half$. It follows that $\Tcal_{\zeta}$  is multiplication with $-\pi\sqrt{-1}$ resp.\  $\pi\sqrt{-1}$ in degree $0$ resp.\ $2$.

The assertion  that  $\Tcal_{\zeta}$  is  multiplication with $2\pi\sqrt{-1}$ on  $H^0(C,\Omega_C)$ follows via the translation into bidifferentials (as described in Remark \ref{rem:bidifftoform}) formally from 
Theorem \ref{thm:lagproj}, but let us spell out the proof nevertheless.
Let $p\in C$ and let $(U;z)$ be a chart centered at $p$ so that $\zeta|U^2$ takes the form \eqref{eqn:chart}.
It is  clear that then $\zeta|C\times U$ is of the form  $\pi_1^*\eta\wedge dz_2$, 
where $\eta$  is a section of $\Omega_{C\times U/U}(2\Delta_U)$.
 Integration of $\eta|U^2$ (which is of the form $(z_1-z_2)^{-2}dz_1$ plus some  holomorphic  relative differential) with respect to  $z_1$ yields $-(z_1-z_2)^{-1}$ plus some holomorphic  function on $U^2$.
Let $\omega$  be an holomorphic $1$-form on $C$. So $\omega|U=g(z)dz$ for some holomorphic $g$.
The residue pairing \eqref{eqn:cauchy} with respect to the first coordinate (with $z_2$ as parameter) yields for any given  $z_2\in U$
\begin{multline*}
\int_C  \omega\wedge[\eta(z_2)] =-\int_C  [\eta(z_2)]\wedge\omega= -2\pi\sqrt{-1}\res_{z_1\to z_2} -(z_1-z_2)^{-1} \omega = \\
=2\pi\sqrt{-1}\res_{z_1\to z_2} (z_1-z_2)^{-1} g(z_1)dz_1= 2\pi\sqrt{-1}g(z_2).
\end{multline*}
This shows that $\Tcal_{\zeta}(\omega)|U=2\pi\sqrt{-1}\omega|U$. It follows that  $\Tcal_{\zeta}(\omega)=2\pi\sqrt{-1}\omega$ everywhere.

We now write the component of $\zeta$ in $H^1(C;\CC)\otimes H^1(C; \CC)$ as
\[
\textstyle \sum_{i,j} \big(a_{ij}\omega_i\otimes \omega_j +b_{ij} \omega_i\otimes \bar\omega_j
+c_{ij}\bar\omega_i\otimes \omega_j+d_{ij}\bar\omega_i\otimes \bar\omega_j\big).
\]
The assumption that  $\zeta$ is $\sigma$-anti-invariant  implies that $a_{ij}$ and $d_{ij}$ are symmetric and that $b_{ij}=c_{ji}$.
From what we proved, it follows that $c_{ij}$ is  the identity matrix and that $d_{ij}=0$.
Since we have the freedom of  choosing  the $a_{ij}$'s arbitrary, the unique representative in question is then obtained
by taking each $a_{ij}=0$. This proves that  the cohomology class $\zeta$  has the stated properties. 
\end{proof}

The Hodge numbers of $C^2$ in degree 2 are 
$h^{2,0}(C^2)=h^{0,2}(C^2)=g(C)^2$ and $h^{1,1}(C^2)=2g(C)^2+2$.
The  subspace $H^{1,1}(C^2)$ is defined over $\RR$ and has (as any K\"ahler surface) Lorentzian signature, so in this case $(1, 2g(C)^2+1)$.

\begin{proposition}\label{prop:}
The self-intersection of the (real) class of $(2\pi\sqrt{-1})^{-1}\zeta_C$ is $2g(C)-1$. In particular, the orthogonal complement of 
this class in $H^{1,1}(C^2; \RR)$ is negative definite when $g(C)$ is positive.
\end{proposition}
\begin{proof}
By Theorem \ref{thm:uchar}, the form  $(2\pi\sqrt{-1})^{-1}\zeta_C$ is the sum of    $(2\pi\sqrt{-1})^{-1}(- \half[C]\otimes 1+1\otimes\half[C])$  and $(2\pi\sqrt{-1})^{-1}\sum_{i=1}^g (\omega_i\otimes \bar\omega_i+\bar\omega_i\otimes \omega_i)$. These two terms are perpendicular to each other. We show that the self-intersection of the first term is 
$-1$ and of the second is $2g(C)$.

We denote the standard intersection form  $H^\pt(C^2)$ (and its complexification $H^\pt(C^2; \CC)$) by 
$(u,v)\mapsto u\cdot v$ and recall that  if $\alpha,\alpha', \beta, \beta' \in H^\pt(C)$ with $\alpha'$ of degree $p$ and  $\beta$ of degree  $q$, 
then  $(\alpha\otimes \alpha')\cdot (\beta\otimes \beta')=(-1)^{pq}( \alpha\cdot \alpha')(\beta\cdot \beta')$. 
Since $(2\pi\sqrt{-1})^{-1}[C]$ represents the natural generator of $H^2(C)$, it is clear that 
\[
\textstyle (2\pi\sqrt{-1})^{-1}[C]\otimes 1)\cdot(1\otimes 2\pi\sqrt{-1})^{-1} [C])=1
\]
This implies that  the self-intersection of $(2\pi\sqrt{-1})^{-1}(- \half[C]\otimes 1+1\otimes\half[C])$ is $-1$.
We next compute 
\begin{multline*}
\textstyle \omega_i\otimes\bar\omega_i\cdot \bar\omega_j\otimes \omega_j=
-(\omega_i\cdot \bar\omega_j) (\bar\omega_i\cdot\omega_j)=\\=(\omega_i\cdot \bar\omega_j) (\omega_j\cdot \bar\omega_i)= 
(2\pi\sqrt{-1})\delta_{i,j}(2\pi\sqrt{-1})\delta_{j,i}=(2\pi\sqrt{-1})^2 \delta_{i,j}.
\end{multline*}
The intersection product on $C^2$ is symmetric and so we get the same value for 
$\bar\omega_j\otimes \omega_j\cdot \omega_i\otimes\bar\omega_i$. It follows that  $(2\pi\sqrt{-1})^{-1}\sum_{i=1}^g (\omega_i\otimes \bar\omega_i+\bar\omega_i\otimes \omega_i)$ has self-intersection $2g(C)$.
\end{proof}

\begin{remark}\label{rem:}
Since $(2\pi\sqrt{-1})^{-1}\zeta_C\cdot 1\otimes [C]=-\half$,  the class of   $(2\pi\sqrt{-1})^{-1}\zeta_C$ is not a K\"ahler class.
\end{remark}

\section{A generalization for higher powers of the curve}\label{sect:higherpowers}
We will here find a higher order generalization of the $2$-form that we introduced in the previous section. 
We here identify $n$ with $\ZZ/n$ so that $C^n$ is  thought of as $C^{\ZZ/n}$.

\subsection{The form $\zeta_n$} Let $\Delta_{ij}$ stand for the diagonal  divisor in $C^n$ defined by $z_i=z_j$ and consider for $n\ge 2$    the divisor 
\begin{gather*}
\textstyle E_n:=\sum_{i\in \ZZ/n} \Delta_{i,i+1},
\end{gather*}
(so that $E_2=2\Delta_{0,1}$).
Note that $E_n$ is a normal crossing divisor  away from the main diagonal. Let $\Delta_{pr} \colon C\hookrightarrow  C^n$ be  the main diagonal. 
By taking successive residues along $\Delta_{0,n-1},\dots \Delta_{2,1}$ and finally a biresidue along $\Delta_{1,0}$, we see that the coherent pull-back $\Delta_{pr}^*\Omega^n_{C^n}(E_n)$ is canonically isomorphic with $\Ocal_C$. We denote the  kernel of the restriction   $\Omega^n_{C^n}(E_n)\to \Ocal_C$ by $ \Omega^n_{C^n}(\log E_n)$. so that we have a short exact sequence
\[
0\to  \Omega^n_{C^n}(\log E_n)\to  \Omega^n_{C^n}(E_n)\to \Delta_{pr *}\Ocal_C\to 0.
\]
The reason for this notation this is that in terms of a local chart $z$ at $p$,
the ideal defining the main diagonal at $\Delta_{pr} (p)$ is generated by $(z_i-z_{i+1})$ and so 
\[
\textstyle \Omega^n_{C^n}(\log E_n)=\sum_{i\in \ZZ/n} \Ocal_{C^n}(-\Delta_{i,i+1})\Omega^n_{C^n}(E_n)=\sum_{i\in \ZZ/n} 
\Omega^n_{C^n}(\sum_{j\not=i} \Delta_{j,j+1}).
\]
Since each sum $D_n^{i,i+1}:=\sum_{j\not=i} \Delta_{j,j+1}$ is a normal crossing divisor, this sheaf consists of logarithmic forms. In order to give this a Hodge theoretic interpretation, we better first blow up $C^n$ along its main diagonal.
Let $f: \tilde C^n\to C^n$ be this blowup and write $\tilde E_n$ for the strict transform of  $E_n$ and $E(f)$ for the preimage of the main diagonal.
Then  $\tilde E_n+ E(f)$  is a normal crossing divisor. We have  $f^*\Omega^n_{C^n}(\log E_n)=
\Omega^n_{\tilde C^n}(\tilde E_n+E(f))$ and   $f^*\Omega^n_{C^n}(E_n)=
\Omega^n_{\tilde C^n}(\tilde E_n + 2E(f_n))$.

Since $\tilde E_n +E(f)$ is a normal crossing divisor,  a theorem of  Deligne \cite{deligne:hodge2}, Th.\ (3.2.5) tells us that  the natural map
\[
H^0(\tilde C^n, \Omega^n_{\tilde C^n}(\tilde E_n+E(f)))\to H^n(\tilde C^n\ssm \tilde E_n)\cong H^n(C^n\ssm E_n)
\]
is injective.

We write  $D_n$ for $D^{0,1}_n=\sum_{i=1}^{n-1} \Delta_{i, i+1}$, so that $E_n=D_n+ \Delta_{n-1,0}(C^{n-1})$.

Let   $H^{\le 1}(C)$ stand for the $H^0(C)\cong\ZZ$-module $H^0(C)\oplus H^1(C)$. We sometimes regard this as the  quotient algebra $H^\pt(C)/H^2(C)$.

\begin{lemma}\label{lemma:string}
The inclusion $C^n\ssm D_n\subset C^n$ identifies $H^\pt (C_n\ssm D_n)$ with the quotient of $H^\pt(C^n)$ by the ideal generated
by the classes $\Delta_{i, i+1!}(1)$ of the irreducible components $\Delta_{i,i+1}$ of $D_n$  ($i=1, \dots, n-1$). So
in the K\"unneth decomposition of $H^\pt(C^n)\cong H^\pt(C)^{\otimes n}$ the subsum of the tensor products
not involving any $\pi_i^*[C]$ with  $i=2,3, \dots, n$, maps isomorphically onto $H^\pt (C^n\ssm D_n)$. To be precise,  the evident algebra homomorphism   $H^\pt (C^n\ssm D_n)\to H^{\le 1}(C)^{\otimes n}$ is surjective and  fits in the exact sequence
\begin{equation}\label{eqn:quotient}
0\to  H^{\le 1}(C)^{\otimes (n-1)}\xrightarrow{\pi_0^*[C]\cup}H^\pt (C^n\ssm D_n)\to H^{\le 1}(C)^{\otimes n}\to 0.
\end{equation}
\end{lemma}

\begin{proof}
For $n=2$ this does not tell us anything new. Indeed, the sequence \eqref{eqn:quotient}  amounts to  the exact sequence \eqref{eqn:simpleshortexactsequence}. We now proceed with induction on $n$.

The closed embedding $C^{n-1}\ssm D_{n-1}\xrightarrow{\Delta'_{n-1, 0}} (C^{n-1}\ssm D_{n-1})\times C$ has  complement $C^n\ssm D_n$. The associated long exact Gysin sequence is 
\[
\cdots \to H^{k-2}(C^{n-1}\ssm D_{n-1})(-1)\xrightarrow{\Delta'_{n-1, 0!}} 
H^{k}((C^{n-1}\ssm D_{n-1})\times C)\to H^k(C^n\ssm D_n)\to\cdots
\]
We regard  $\Delta'_{n-1, 0!}$ as a  homomorphism of $H^\pt(C^{n-1}\ssm D_{n-1})$-modules. Since  $\Delta'_{n-1, 0}$ is a section of the projection $(C^{n-1}\ssm D_{n-1})\times C\to C^{n-1}\ssm D_{n-1})$, the map
$\Delta'_{n-1, 0!}$ is injective (its left inverse is integration along the fibers) and so this Gysin sequence splits up is short exact sequences. 
It follows that  $\Delta_{n-1, 0!}$  identifies $ H^{\pt-2}(C^{n-1}\ssm D_{n-1})(-1)$ with a graded ideal in  the graded algebra 
$H^{\pt}(C^{n-1}\ssm D_{n-1})\otimes H^\pt( C)$ with quotient  $H^\pt(C^n\ssm D_n)$.
Our induction hypothesis then implies that $H^\pt(C^n\ssm D_n)$ is as asserted.
\end{proof}

The map  $H^k(C^n)\to H^k(C^n\ssm D_n)$ is a morphism of mixed Hodge structures.  By Lemma \ref{lemma:string} it is onto and so $H^k(C^n\ssm D_n)$ has pure weight $k$ and the Hodge filtration of $H^k(C^n\ssm D_n)$ is the image of the Hodge filtration of  $H^k(C^n)$. Hence the induced map
\[
\Gr^p_FH^k(C^n)=\Gr^p_F\Gr_k^{W}H^k(C^n)\to \Gr^p_F\Gr_k^{W} H^k(C^n\ssm D_n)=\Gr^p_FH^k(C^n\ssm D_n)
\]
is also surjective. Since $D_n$ is a  normal crossing divisor, this  is by Deligne \cite{deligne:hodge2} Cor.\ (3.2.13) just the map $H^{k-p}(C^n, \Omega^p_{C^n})\to H^{k-p}(C^n, \Omega^n_{C^n}(D_n))$. For $k=p$, this map  is clearly injective and so we  find:

\begin{corollary}\label{cor:string}
The inclusion  $\Omega^p_{C^n}\subset\Omega^p_{C^n}(D_n)$ induces a surjection on cohomology and is an isomorphism in degree $0$. \hfill$\square$
\end{corollary}

The following proposition may be regarded as a generalization of  the construction of our canonical bidifferential on $C^2$.

\begin{theorem}\label{thm:cyclicclass}
For $n\ge 2$, the space  $H^0(C^n,\Omega^n_{C^n}(E_n))$ embeds in $H^n(C^n\ssm E_n; \CC)$ and lands in $F^{n-1}H^n(C^n\ssm E_n)$. The inclusion $\Omega^n_{C^n}\subset \Omega^n_{C^n}(E_n)$ gives rise  to a  short exact sequence
\[
0\to H^0(C^n,\Omega^n_{C^n})\to H^0(C^n,\Omega^n_{C^n}(E_n))\to \CC\to 0, 
\]
where the map $H^0(C^n,\Omega^n_{C^n}(E_n))\to \CC$ is given as an iterated (bi)residue:
\[
\zeta\in H^0(C^n,\Omega^n_{C^n}(E_n))\mapsto \Bir_{\Delta_{1,0}}\res_{\Delta_{2,0}}\cdots \res_{\Delta_{n-2,0}}\res_{\Delta_{n-1,0}} \zeta\in  \CC.
\]
There  is a unique   $\zeta_n\in H^0(C^n,\Omega^n_{C^n}(E_n))$ defining an element of $H^n(C^n\ssm E_n; \RR(1-n))$
(of Hodge type $(n-1, n-1)$) which maps to $1$.  For $n\ge 3$, $\res_{D_n}\zeta_n=\zeta_{n-1}$ (where $\zeta_2=\zeta_C$ is the $2$-form defined earlier). 
In  terms of a local coordinate $z$ on $C$, the  polar part of $\zeta_n$ along the main diagonal is 
\begin{equation}\label{eqn:polarpart}
\frac{dz_{n-1}\wedge\cdots \wedge dz_0}{(z_0-z_1)(z_1-z_2)\cdots (z_{n-2}-z_{n-1})(z_{n-1}-z_{0})}.
\end{equation}
\end{theorem}

\begin{remark}\label{rem:}
The $\Sfrak_n$-stabilizer of $E_n$ is a dihedral group $\Dfrak_n$ of order $2n$. The uniqueness assertion implies that 
 $\zeta_n$ transforms under this stabilizer  according to the character $\chi_n\colon \Dfrak_n\to \{\pm 1\}$ defined by the expression \eqref{eqn:polarpart}.  The restriction of  $\chi_n$ to the normal subgroup of rotations in $\Dcal_n$ (a cyclic subgroup of order $n$) is clearly the restriction of the sign character of $\Sfrak_n$;  the value of $\chi_n$ on a reflection is $(-1)^n$ times the sign character and hence  given  by the 
 parity of $n(n+1)/2$.
 \end{remark}

\begin{remark}\label{rem:}
For distinct points $p_1, \dots , p_k$ of $C$, 
our $\zeta_{k+1}$  produces a canonical  linear embedding  
\[
\textstyle T_{p_1}C\otimes \cdots \otimes T_{p_k}C\hookrightarrow H^0(C, \Omega_C(\sum_{i=1}^kp_i))
\]
which depends holomorphically on $(p_1, \dots, p_k)\in \conf_k(C)$ (but will in general not holomorphically depend  on $C$ if $C$ moves in a holomorphic family).
\end{remark}

\begin{proof}[Proof of Theorem \ref{thm:cyclicclass}]
We already established this for $n=2$. We  proceed with induction on $n$, so  assume $n>2$ and that  the proposition has been verified for $n-1$.
Taking the residue along $\Delta_{n-1,0}$ gives for $n\ge 3$ the exact sequence
\[
0\to \Omega^n_{C^n}(D_n)\to\Omega^n_{C^n}(E_n)\to \Delta_{n-1,0 *}\Omega^{n-1}_{C^{n-1}}(E_{n-1})\to 0
\]
whose associated  long exact sequence begins with 
\[
0\to H^0(C^n, \Omega^n_{C^n}(D_n))\to H^0(C^n, \Omega^n_{C^n}(E_n))\to H^0(C^{n-1},\Omega^{n-1}_{C^{n-1}}(E_{n-1}))\to H^1(C^n,\Omega^n_{C^n}(D_n)) 
\]
This sequence  is compatible with the Gysin sequence for the closed embedding $\Delta_{n-1,0}\colon C^{n-1}\ssm E_{n-1}\hookrightarrow C^n\ssm D_n$ (with complement  $C^n\ssm E_n$) in the sense that we have almost a  morphism (=commutative diagram) of exact sequences
\[
\begin{tikzcd}[column sep=tiny]
&H^n(C^n\ssm D_n;\CC)\arrow[r]& H^n(C^n\ssm E_n; \CC)\arrow [r]& H^{n-1}(C^{n-1}\ssm E_{n-1}; \CC)(-1)\arrow[r] & 
H^{n+1}(C^n\ssm D_n; \CC)\\
0\arrow[r] &H^0(C^n,\Omega^n_{C^n}(D_n))\arrow[r]\arrow[u] & H^0(C^n,\Omega^n_{C^n}(E_n))\arrow [r, "\res_{\Delta_{n-1,0}}"]\arrow[u] & H^0(C^{n-1},\Omega^{n-1}_{C^{n-1}}(E_{n-1}))\arrow[r]\arrow[u] & H^1(C^n,\Omega^n_{C^n}(D_n))\arrow[u, dashed]
\end{tikzcd}.
\]
We say almost, because at this point, it is not yet clear how the dashed  arrow is defined (and is then such that the square containing it commutes). We can  invoke  Corollary \ref{cor:string} to give it a sense: it will then be injective with image $H^{n,1}(C^n\ssm D_n)$. To proceed more formally, we must divide  out the top sequence by its $F^{n+1}$-part; it will then stay exact because of the strictness property of the Hodge filtration. Since $F^{n+1}H^n(C^n\ssm D_n)=0$ and $F^{n+1}\big(H^{n-1}(C^{n-1}\ssm E_{n-1})(-1))=
\big( F^nH^{n-1}(C^{n-1}\ssm E_{n-1})\big)(-1)=0$ by induction, it follows that  $F^{n+1}H^n(C^n\ssm E_n)=0$ and so 
this operation  only affects  its last term $H^{n+1}(C^n\ssm D_n; \CC)$.
This turns the dashed map into an  isomorphism of    
$H^1(C^n,\Omega^n_{C^n}(D_n))$ onto the subspace $\Gr^n_F H^{n+1}(C^n\ssm D_n)$ of $H^{n+1}(C^n\ssm D_n; \CC)/F^{n+1} H^{n+1}(C^n\ssm D_n)$.

Two other up arrows also have  a  mixed Hodge theory interpretation: 
$H^0(C^n,\Omega^n_{C^n}(D_n))$ gets identified with 
$F^n\cap W_n$ of its target (it is in fact pure of type $(n,0)$) and  our induction  says that $H^0(C^{n-1},\Omega^{n-1}_{C^{n-1}}(E_{n-1}))(-1)$ splits into $H^0(C^{n-1},\Omega^n_{C^{n-1}})(-1)$ (which maps to $F^n\cap W_n$) and the span of  $\zeta_{n-1}$ (which maps to  a class of type $(n-1,n-1)$). In particular, $\zeta_{n-1}$ dies in $H^1(C^n,\Omega^n_{C^n}(D_n))$. So it has a preimage $\zeta_n$ in $H^0(C^n,\Omega^n_{C^n}(E_n))$. Since 
$H^0(C^n,\Omega^n_{C^n}(D_n))$ is pure of type $(n,0)$, we  can take this preimage to be in  
$F^{n-1}\cap W_{2n-2}$  (of type $(n-1,n-1)$) and it is  then unique.
Since  $\zeta_{n-1}\in H^{n-1}(C^{n-1}\ssm E_{n-1}; \RR(2-n))$, it follows that $\zeta_n\in H^n(C^n\ssm E_n; \RR(1-n))$.

It is clear that near in terms of a local coordinate on $C$, the form $\zeta_n$ will be near the main diagonal as stated.  
\end{proof}

\begin{proposition}\label{prop:}
The form $\zeta^{(n)}:=f^*\zeta_n$ naturally  defines  an element of $H^n(\tilde C^n\ssm \tilde E_n; \CC)$.
\end{proposition}
\begin{proof}
We use on the projectivized normal bundle of the principal diagonal the projective coordinates  $[z_1-z_0:\cdots : z_{n-1}-z_{n-2}:z_0-z_{n-1}]$ subject to the evident condition that their sum is zero. The affine open subset  defined by $z_1\not=z_0$ is therefore parametrized  by $[1:v_1:\cdots v_{n-2}:v_{n-1}]$, where $1+v_1+\cdots +v_{n-1}=0$. 
If we put $u:=z_1-z_0$, then $f^{-1}U^n\ssm \tilde\Delta_{1,0}$  has near $E(f)$ the coordinate system $(u,z_0,v_1,\dots, v_{n-2})$, for then $z_i=z_0+u(v_1+\cdots +v_i)$, $i=1, \dots, n-2$ and since  $v_1+\cdots +v_{n-1}=-1$, we also have that $z_{n-1}=z_0-u$. Note that   $E(f)$ is defined by $u=0$.
In terms of these coordinates,  the polar part of $f^*\zeta_n$ is given by
\begin{multline*}
f^*\frac{dz_{n-1}\wedge\cdots \wedge dz_0}{(z_{n-1}-z_{n-2})\cdots (z_1-z_0)(z_0-z_{n-1})}=\\=
\frac{d(z_0-u)\wedge d(z_0+u(v_1+\cdots +v_{n-2}))\wedge\cdots\wedge d(z_0+u(v_1+v_2))\wedge d(z_0+uv_1)\wedge dz_0 }{u^{n}(-1-v_1-\cdots-v_{n-2})v_{n-2}\cdots v_{1}}=\\
=\frac{d(-u)\wedge udv_{n-2}\wedge\cdots\wedge udv_2\wedge ud v_1\wedge dz_0 }{u^{n}(-1-v_1-\cdots-v_{n-2})v_{n-2}\cdots v_{1}}=
\frac{du\wedge dv_{n-2}\wedge\cdots\wedge d v_1\wedge dz_0 }{u^{2}(1+v_1+\cdots+v_{n-2})v_{n-2}\cdots v_{1}}.
\end{multline*}
We can write this as $d\eta$, where 
\[
\eta =-\frac{dv_{n-2}\wedge\cdots\wedge d v_1\wedge dz_0 }{u(1+v_1+\cdots+v_{n-2})v_{n-1}\cdots v_{1}}.
\]
So $f^*\zeta_n$ is exact on a neighborhood of $E(f)\ssm  \tilde E_n)$ in $\tilde C^n\ssm \tilde E_n$

This  formally  implies that  $f^*\zeta_n$ lands in $H^n(\tilde C^n\ssm \tilde E_n)$. To make this explicit, let $\varphi \colon C^n\to [0,1]$ be a smooth function whose support is a  tubular neighbourhood of the principal diagonal and which is constant 1 
on a smaller tubular neighbourhood of the same. Then  $f^*\zeta_n-d(\varphi\eta)=f^*\zeta_n-\varphi d\eta -d\varphi\wedge \eta$ is regular on $\tilde C^n\ssm \tilde E$ and represents  there a cohomology class. 
\end{proof}

\begin{question}\label{quest:}
Which cohomology class in $H^n(\tilde C^n\ssm \tilde E_n; \CC)$ does $\zeta_n$ represent? For example, if $i\in \ZZ/n\mapsto a_i\in H_1(C)$  is such that $a_i\cdot a_j=0$ for all $i,j$, then we can represent $a_i$ by an  $1$-cycle $A_i$ such  that $A_0\times \cdots \times A_{n-1}$ has its support in $\conf_n(C)$  and hence we can integrate 
$\zeta_n$ over this cycle. What is its value? (Probably this integral makes sense for any element of $\sym^n H_1(C)$.)
\end{question}

\section{Canonical construction of  the Virasoro algebra}
In  this section we return to the setting of Section \ref{sect:lagproj} and work over a field $k$ of characteristic  zero.
We show how the notions  introduced there play with  the Fock and the Virasoro representation. 
We begin  with recalling a basic construction. 

\subsection{Review of the basic Fock  representation}\label{subsect:oscreview}
Let $H$  a finite dimensional $k$-vector space  endowed with a symplectic (=nondegenerate, antisymmetric bilinear) form $(\alpha, \beta)\in H\times H\mapsto \alpha\cdot\beta\in k$. 
Its  dimension is then even. 

An isomorphism of $k$-vector spaces 
\[
T_H:  H\otimes_k H\cong \End_k(H) 
\]
is defined by assigning  to $\alpha\otimes \alpha'\in H\otimes_k H$ the elementary endomorphism 
$x\mapsto 2\alpha (\alpha'\cdot x)$ (the reason for  including the factor $2$ will become clear shortly). This makes composition in $\End_k (H)$ correspond to twice  a contraction: 
\[
T_H(\alpha\otimes \alpha') T_H(\beta\otimes \beta')= 2T_H(\alpha\otimes (\alpha'\cdot\beta)\beta').
\]
The $k$-Lie subalgebra of $\splie (H)\subset \End_k (H)$ of  endomorphisms that infinitesimally preserve the symplectic form correspond under $T_H^{-1}$ with the symmetric  tensors with  the Lie bracket given by 
\[
[\alpha^2, \beta^2]= 2\alpha\otimes (\alpha\cdot \beta) \beta -
2\beta\otimes ( \beta\cdot\alpha)\alpha= 2(\alpha\cdot \beta)(\alpha\otimes \beta+  \beta\otimes\alpha)
\]
Here we write  $\xi\otimes\xi\in H\otimes H$ as $\xi^2$ (it is a quadratic form on the dual of $H$). 
From now on we reserve the notation $T_H$ for the restriction $\sym^2 H\cong \splie (H)$. So $T^{-1}_H$ assigns to  the elementary endomorphism $x\mapsto (x, \alpha)\alpha$ the quadratic function $\half\alpha^2$. 

With the  symplectic vector space $H$  is associated a Heisenberg Lie algebra $\hat H$ whose underlying vector space is  $k\hbar \oplus H$ (\footnote{Here $\hbar$ is just the name of a generator of a 1-dimensional vector space; we could also call it $q$ and in any case, we will soon regard it as an invertable  variable.}) and whose Lie bracket is given by   $[\hat\alpha, \hat\beta]=(\alpha\cdot\beta) \hbar$. So we have an extension of abelian Lie algebra's
\[
0\to k\hbar \to \hat H \to  H\to 0.
\]
We denote the preimage of linear subspace $I\subset H$ by $\hat I$. It is an abelian subalgebra  if and only if $I$ is 
isotropic for the symplectic form (and then $I$ is a subalgebra as well). 

The  universal enveloping algebra 
$\Uscr (\hat H)$ of $\hat H$ is the tensor algebra on $H$ tensored with $k[\hbar]$ modulo the two-sided ideal generated by the tensors 
$\alpha\otimes\beta-\beta\otimes \alpha-(\alpha\cdot\beta)\hbar$. Note that this ideal contains the antisymmetric tensors in $H\otimes H$.
The Poincar\'e-Birkhoff-Witt filtration filtration on  $\Uscr (\hat H)$ is an  increasing filtration  $F_\pt\Uscr (\hat H)$  by finite dimensional subspaces:  here $F_d\Uscr (\hat H)$ is the image of  the  tensors of degree $\le d$. The associated graded algebra $\Gr_F^\pt \Uscr (\hat H)$ is naturally identified with the symmetric  algebra on $\hat H$. 
We denote the product in  $\Uscr (\hat H)$ by $\circ$. 

Here are some useful identities in $\Uscr (\hat H)$. It is straightforward to verify that if $\alpha, \beta, \xi\in H$, then 
$[\alpha\circ\beta, \xi]= (\alpha\cdot\xi)\beta\circ \hbar+(\beta\cdot\xi)\alpha\circ \hbar$.
The right hand side only depends on the image of $\alpha\otimes\beta$  in the symmetric quotient $\sym_2 H$ of $H\otimes H$ (as is to be expected). In particular,  
\begin{equation}\label{eqn:squarerep}
[\alpha\circ\alpha, \xi]= 2(\alpha\cdot\xi)\alpha .\hbar, \quad 
[\alpha\circ\alpha , \beta\circ\beta]= 2(\alpha\cdot\beta)(\beta\circ\alpha +\alpha\circ\beta).\hbar.
\end{equation}
 The factor 2 appearing on the right hand side explains the inclusion of that factor  in the definition of $T_H$. It suggests that we localize as to make $\hbar$ invertible and define a $k$-linear map 
\[
T_\Uscr: \sym^2 H\to \Uscr (\hat H)[1/\hbar], \quad  \alpha^2\mapsto \alpha\circ \alpha /\hbar, 
\]
for then 
\[
[T_\Uscr(\alpha^2), \xi]= 2\alpha(\alpha\cdot \xi)=T_H(\alpha^2)(\xi).
\]
This formula also shows that $T_\Uscr$ defines via $T_H^{-1}$ a representation  of $\splie(H)$ on $\Uscr (\hat H)[1/\hbar]$, for 
\[
[T_\Uscr(\alpha^2), T_\Uscr(\beta^2)]=2(\alpha\cdot\beta)(\alpha\circ\beta +\beta\circ \alpha)/\hbar= T_\Uscr([\alpha^2, \beta^2])
\]

Assume further given a Lagrangian subspace $I\subset H$ (so $I$ is isotropic for the symplectic form and has the  maximal dimension, namely $\half\dim H$,  for this property).
As noted above,  $\hat I$ is then  an abelian subalgebra of $\hat H$. We regard the projection $\hat I=k\hbar\oplus I\to  
k\hbar\cong k$ as a  one-dimensional representation of $\hat I$. By inducing this representation up to $\hat H$ we get the \emph{Fock representation} of $\hat H$, 
\[
\FF(H, I):= \Uscr (\hat H)\otimes_{\Uscr (\hat I)} k
\]
So $\hbar $  acts on $\FF(H, I)$ as the identity.  Let $\mathbf{1}\in \FF(H, I)$ (or $\mathbf{1}_{\FF(H, I)}$ if there is any need to be that precise) stand for $1\otimes_{\Uscr (\hat I)} 1$. It is clear that this element is killed by the abelian subalgebra $I\subset \hat H$ and  generates $\FF(H, I)$ as a representation of $\hat H$. 
Furthermore, each subspace $F_d\Uscr(\hat H).\mathbf{1}$ is preserved by $I$ and killed by $\sym_d(I)$
and hence $I$ acts locally nilpotently on $\FF(H, I)$. 

Note that if $L\subset H$ is a Lagrangian supplement of $F$ in $L$, then its symmetric algebra $\sym_\pt L$ embeds in
$\Uscr(\hat H)$ and the map $u\in \sym_\pt L\mapsto u\circ \mathbf{1}\in \FF(H, I)$ is a $k$-linear isomorphism.

\subsection{The  Fock space for a DVR} 
We here return to the setting  of  Subsection \ref{subsect:bires}. We shall freely  use the notation introduced there. 

The $K$-vector space  $\theta_K$ of $k$-derivations $K\to K$ is closed under the Lie bracket and thus has the structure of a topological $k$-Lie algebra. In terms of a uniformizer $t$, a topological $k$-basis of $\theta_K$ is the collection 
\[
D_n:=t^{n+1}\frac{d}{dt},\quad  n\in \ZZ, 
\]
and we have $[D_n, D_m]=(m-n)D_{n+m}$. The $K$-linear pairing
\[
\theta_K\otimes_K \Omega_K^{\otimes 2}\cong \Omega_K\xrightarrow{\res} k, 
\]
of which the first map is $K$-linear contraction,  takes on $(D_n, t^{m-2}dt^2)$ the value $\delta_{n+m,0}$. 
So this is nondegenerate as a topological $k$-pairing and  identifies $\theta_K$ with the continuous $k$-dual of 
$\Omega_K^{\otimes 2}$.  Consider the short exact sequence \eqref{eqn:biresseq1}. We can replace in that sequence 
$\Omega^{\otimes 2}$ by $\Omega_K^{\otimes 2}$ and get a similar exact sequence 
\[
0\to\Omega_K^{\otimes 2}\to \hat\Omega_K^{\otimes 2}\xrightarrow{\Bir} k\to 0
\]
The middle term can be obtained as a reduction of $\Ical_\Delta^{-2}\Omega_{K^{(2)}}^{(2)}$. Concretely, 
elements of the middle term can be represented by expressions  $c(t_1-t_2)^2 dt^2+fdt^2$ with $c\in k$ and $f\in K$.
Hence the  continuous $k$-dual of the   sequence \eqref{eqn:biresseq2} has the form
\begin{equation}\label{eqn:vir1}
0\to k\to \hat\theta_K\to \theta_K\to 0,
\end{equation}
where $1\in k$ represents the biresidue map $\hat\Omega_K^{\otimes 2}\to k$.
We will show that this exact sequence lives in a natural manner in the category of topological  $k$-Lie algebra's.
making it a central extension of the Lie algebra  $\theta_K$.  The extension will be nontrivial  and thus  produce in a natural manner the Virasoro algebra. 

A first step makes a  connection  with the Heisenberg algebra defined above.
We do this by emulating the construction in Subsection \ref{subsect:oscreview} for this situation, but in what follows we often prefer to work with  the degenerate $K$ rather than its nondegenerate  quotient $K'$. 
Most of  the construction  given in the finite dimensional setting goes through. So we have  a central (Heisenberg) Lie algebra's  $\hat K$ and $\hat K'$  (with underlying vector spaces $k\hbar\oplus K$ and $k\hbar\oplus K'$) whose  centers are  $k\hbar\oplus k$ and  $k\hbar$ respectively.
We  have an Heisenberg algebra 
\[
0\to k\hbar\to \hat K\to K\to 0,
\]
whose Lie bracket is given by  $[\hat f ,\hat g]:=\res (f\, dg)\hbar$ and likewise for $\hat K'$. Since $\Ocal$ is a Lagrangian subspace $K$, we  have an associated  Fock representation of $\hat H$, 
\[
\FF(K):=\FF(K, \Ocal)=\Uscr(\hat K)\otimes _{\Uscr (\hat\Ocal)} k.
\]
The subspace $k\subset K\subset \hat K$ induces the zero map in $\FF(K)$ and so there is no need to introduce $\FF(K')$, because $\FF(K, \Ocal)$ is  equal  to $\FF(K', \mfrak')$.  Note that the abelian subalgebra 
$\mfrak\subset\hat K$ acts on $\FF(K)$ in a locally nilpotent manner: for every $v\in \FF(K)$, $\mfrak^nv=0$ for some $n\ge 0$. 

\subsection{A canonical description of the  Virasoro algebra}
Any  $D\in\theta_K$ defines a $K$-linear map $\Omega_K\to K$. The residue pairing is topologically perfect: it identifies the continuous $k$-dual of $\Omega_K$ with $K$. With the help of the  identity $\res (Df .dg)+\res (Dg . df)=0$, which   expresses the fact that  
$D$ infinitesimally preserves the residue pairing (see Remark \ref{rem:symplecticlie}), one may identify  $\theta_K$ with the completion of the symmetric tensors  in $K\otimes_kK$ and we could be tempted to think of this being an appropriate  completion of the universal enveloping algebra of $\hat K$. So  $T_{K}^{-1}$ would for example make 
$D_n$ correspond to the infinite sum $\half\sum_{i+j=n}  t^i\circ t^j $. There is however no obvious way to let such a sum act on our Fock space $\FF(K)$. It is for this reason that we resort to the normal ordering convention  discussed below. This depends on the choice of a Lagrangian projection  $K'\to \mfrak'$.

We begin with choosing a   $\eta\in (\hat\Omega^{(2)})^{\sigma}$ with biresidue $1$. We have seen (Proposition \ref{prop:Lagrangiansupp}) that  this  determines  a Lagrangian supplement $L_\eta$ of $\Ocal$ in $K$: 
if $\eta$ is written as 
\[
\textstyle \eta=(t_1-t_2)^{-2}dt_1 dt_2 +\eta_0 \quad \text{with}\quad \eta_0=\sum_{n\ge 1}  t_1^{n-1}dt_1.\pi_2^*(df_n),\quad\text{where}\quad  f_n\in \Ocal,
\]
then $L_\eta$ has a $k$-basis $\{t^{-n}-f_n\}_{n=1}^\infty$.
This identifies $\FF(K)$ with $\sym_\pt L_\eta$ as a $k$-vector space. 
This splitting  also gives rise to the following convention:
if $f\in K$ is written $f_-+f_+$ with $f_-\in L_\eta$ and $f_+\in\Ocal$ and   $g\in  K$ is written likewise as $g_-+g_+$, then  the \emph{normally ordered tensor product} (associated with $\eta$)  is 
\[
f\circ_\eta g= f_-\circ g_-  +f_-\circ g_+ + f_+\circ g_+ +  g_-\circ f_+, 
\]
which we  here regard as an element of in the universal enveloping algebra $\Uscr(\hat K)$ (\footnote{In the literature this is usually denoted
$\colon\! f\circ g\colon$ of a variation thereof; our  notation wants to emphasize the dependence on $\eta$.}). So we only reversed the order on $\Ocal\otimes L_\eta$:  we replaced $f_+\circ g_-$  by $g_-\circ f_+$.
Hence the right hand side is also equal to $f\circ g- f_+\circ g_-+ g_-\circ f_+=f\circ g-\la f_+|g_-\ra \hbar$ and  so 
$f\circ_\eta g$ differs from $f\circ g$  by a central element. 

Concretely, if $f_n\in\Ocal$ is as above, then 
\[
(t^n)_+=
\begin{cases}
f_{-n} & \text{if $n<0$;}\\
t^{n} & \text{if $n\ge 0$,}
\end{cases}
\quad\text{and}\quad 
(t^n)_-=
\begin{cases}
t^n-f_{-n}& \text{if $n<0$;}\\
0 & \text{if $n\ge 0$.}
\end{cases}
\]
and for integers $i,j$,
\[
t^i\circ_\eta t^j=
\begin{cases}
t^i\circ t^j-\la f_{-i}| t^j\ra\hbar & \text{if $i<0$ and $j<0$;}\\
t^j\circ t^i  & \text{otherwise.}\\
\end{cases}
\]

We shall see that things begin to diverge (in more sense than one), if we complete and consider elements on $K^{(2)}$. 

Given a $D\in \theta_K$, let  $T_\eta(D)$ stand for the associated symmetric  tensor written  \emph{in normally ordered form} in a suitable completion of $\Uscr(\hat K)[1/\hbar]$. For example, 
\[
\textstyle T_\eta(D_n):=\frac{1}{2\hbar}\sum_{i+j=n}  t^i\circ_\eta t^j
\]
(for the coefficient $1/(2\hbar)$, see  the identities \eqref{eqn:squarerep}). If we work this out, we see that
\begin{equation}\label{eqn:T(D)}
T_\eta(D_n)=
\begin{cases}
\frac{1}{2\hbar}t^{n/2}\circ t^{n/2} + \frac{1}{\hbar}\sum_{j>n/2} t^{n-j}\circ t^j & \text{if $n\ge 0$}; \\
\frac{1}{2\hbar} t^{n/2}\circ t^{n/2} + \frac{1}{\hbar}\sum_{j>n/2} t^{n-j}\circ t^j -\sum_{i=1}^{1-n} \la f_i|t^{n+i}\ra &\text{if $n<0$.}
\end{cases}
\end{equation}
Here $t^{n/2}\circ t^{n/2}$ should be read as zero in case $n$ is odd.

The  action of $\sum_{i+j=n} t^i\circ_\eta t^j$ on $\FF(K)$ now makes  sense, since   for  any $v\in \FF(K)$, $t^r(v)=0$ when $r$ is sufficiently large, say $r>r_0$  and so  $T_\eta(D_n)(v)$ is a finite sum which only involve terms with $n/2\le j\le r_0$. The element $\hbar$ acts as the identity on $\FF(K)$. An arbitrary $D\in \theta_K$ is written $\sum_{n\ge r} c_nD_n$ with $c_n\in k$, and so we see that its action on $\FF(K)$ differs from the one for which each $f_n$ is zero (in other words, for the case when $\eta_0=0$) by a multiple the identity, this multiple being the finite sum
 $\sum_{0<n<-r} -c_{-n}\sum_{i=1}^{1-n} \la f_i|t^{n+i}\ra$. We thus have defined a  $k$-linear map 
$T_\eta\colon \theta_K\to \Uscr(\hat K)[1/\hbar]$. 

\begin{proposition}\label{prop:vir}
The map $T_\eta\colon \theta_K\to \Uscr(\hat K)[1/\hbar]$ satisfies 
\begin{enumerate}
\item[(i)] $[T_\eta(D),\hat f]= D(f)$ (where $\hat f\in \hat K$) and
\item[(ii)] $[T_\eta(D_k), T_\eta(D_l)]=(l-k)T_\eta(D_{k+l})+ \frac{1}{12}(k^3-k)\delta_{k+l,0}$.
\end{enumerate}
\end{proposition}

\begin{proof}
For the verification of these Lie brackets, we can, when using the formula's \eqref{eqn:T(D)} ignore multiples of the identity. In other words, there is no loss in generality in assuming  that each $f_n$ is zero (in other words that $\eta_0=0$). The computations are then straightforward (see for example Lecture 3 in \cite{kacraina}).
\end{proof}

Property (i) of Proposition \ref{prop:vir} is exactly as in the finite dimensional case 
and  also  shows that $T_\eta$ is injective.  But property  (ii) is a bit different: the Lie bracket is preserved up to 
scalar term, which means that $T_\eta$ is only a projective-linear  representation of  Lie algebra's. 
We can do better, though. Proposition \ref{prop:vir} shows that the $k$-span of identity in $\Uscr(\hat K)[1/\hbar]$ and the image  of $T_\eta$ is a Lie subalgebra of $\Uscr(\hat K)[1/\hbar]$, which is independent of $\eta$. We may call this the Virasoro algebra, but since we prefer the reserve that term for something naturally isomorphic to it, let us refer  to this as the \emph{Virasoro subalgebra}  of $\Uscr(\hat K)[1/\hbar]$ and denote it by $\Vir(\FF(K))$.

\begin{theorem}[Canonical form of the Virasoro algebra]\label{thm:vir}
The  continuous $k$-dual of the exact sequence $0\to \Omega_K^{\otimes 2}\to \hat\Omega_K^{\otimes 2}\to k\to 0$, which is an exact sequence of the form  $0\to k\to \hat\theta_K\to \theta_K\to 0$  has the property that the middle term is naturally  identified with the Virasoro subalgebra $\Vir_\Uscr(K)\subset \Uscr(\hat K)[1/\hbar]$. This  makes 
$\hat\theta_K$ in a canonical manner a central  extension of Lie algebra's.
\end{theorem}
\begin{proof}
We first extend $T_\eta$ to a $k$-linear isomorphism  $\hat T_\eta: \hat\theta_K\xrightarrow{\cong} \Vir_\Uscr(K)$  as follows: use the image of $\eta$ in $\hat\Omega^{\otimes 2}$ to split the sequence  \eqref{eqn:vir1} so that we have a $k$-linear isomorphism $\hat\theta_K\cong k\oplus \theta_K$ and then  let  $\hat T_\eta$ take the element of $\hat\theta_K$ defined by $(\lambda, D)\in k\oplus\theta_K$ to $T_\eta(D)+\lambda.1$. We must show that this isomorphism does not depend on $\eta$.

This is a consequence of  equivariance with respect to the  actions of the vector group  $(\Omega^{(2)})^\sigma$.
We first regard  $(\Omega^{(2)})^\sigma$ as an abelian Lie subalgebra of the symplectic Lie algebra of $K$ by  
assigning to  $\xi\in (\Omega^{(2)})^\sigma$ the continuous $k$-linear map  $f\in K\mapsto \res_1 \pi_1^*(f)\xi\in \Omega$ followed by the identification $\Omega\cong \mfrak\subset K$ by means of $d^{-1}$. This action then integrates to a  continuous unipotent action of the vector group $(\Omega^{(2)})^\sigma$ on $K$ by 
\[
\rho(\xi): f\in K\mapsto f+ d^{-1}(\res_1 \pi_1^*(f)\xi)\in K
\]
This transformation is the identity on $\Ocal$ and  if $\xi$ is written 
$\sum_{n\ge 1} t_1^{n-2}dt_1.\pi_2^*dg_n$ with $g_n\in \mfrak$, then $\rho(\xi)$ takes $t^{-n}$ to $t^{-n}+g_n$ $(n>0)$.
In particular, it takes $L_\eta$ to $L_{\eta+\xi}$. Similar, it takes the splitting of $\hat\theta_K$ defined
by $\eta$  to the splitting defined by $\eta+\xi$.
This action also respects the symplectic form and  hence extends to one on $\Uscr(\hat K)[1/\hbar]$.
Formula's \eqref{eqn:T(D)} make it clear how  $\rho(\xi)$ affects $T_\eta (D_n)$ (replace $f_i$ by $f_i+g_i$)
and so this is just $T_{\eta +\xi}(D_n)$. 
This proves that $\hat T_\eta$ is independent of $\eta$.
\end{proof}

\end{document}